\documentclass[10]{article}
\usepackage{graphicx}
\usepackage{amssymb,amsmath,amsthm,epsfig}
\usepackage{color}
\usepackage[T1]{fontenc}
\usepackage[british]{babel}
\usepackage{color}
\usepackage{amsmath}
\usepackage{amsthm}
\usepackage{amsfonts}
\usepackage{amssymb}
\usepackage{mathrsfs}
\usepackage{amscd}
\usepackage{dcolumn}
\usepackage{type1cm}
\usepackage{xspace}
\usepackage{multirow}
\newtheorem{theorem}{Theorem}
\newtheorem{lemma}{Lemma}
\newtheorem{corollary}{Corollary}

\newtheorem{proposition}{Proposition}

\newtheorem{remark}{Remark}
\newcommand{\K}{\mathbf K}

\newcommand{\E}{\mathbf E}

\newcounter{rea}
\newcounter{rek}
\definecolor{Plum}{cmyk}{0.50,1.,0,0}
\definecolor{myred}{rgb}{0.9,0,0.1}
\definecolor{myorange}{rgb}{0.8,0.3,0.1}

\usepackage{ulem}  

\begin{document}


%
\begin{center}
{\large {\bf Uniform approximation and explicit estimates for the prolate spheroidal wave functions.}}\\
\vskip 1cm Aline Bonami$^a$ and Abderrazek Karoui$^b$ {\footnote{
This work was supported in part by the  ANR grant "AHPI" ANR-07-
BLAN-0247-01, the French-Tunisian  CMCU 10G 1503 project and the
DGRST  research grant 05UR 15-02.\\
Part of this work was done while the second author was visiting
the research laboratory MAPMO of the University of Orl\'eans,
France.}}
\end{center}
\vskip 0.5cm {\small
$^a$ F\'ed\'eration Denis-Poisson, MAPMO-UMR 6628, d\'epartement de math\'ematiques, Universit\'e d'Orl\'eans, 45067 Orl\'eans cedex 2, France.\\
\noindent $^b$ University of Carthage, Department of Mathematics, Faculty of Sciences of Bizerte, Tunisia.}\\
Email: aline.bonami@univ-orleans.fr (A. Bonami), abderrazek.karoui@fsb.rnu.tn (A. Karoui)\\

\noindent{\bf Abstract}--- For fixed $c,$ Prolate Spheroidal Wave
Functions (PSWFs), denoted by $\psi_{n, c},$ form an orthogonal  basis with remarkable properties
for the space of band-limited functions with bandwith $c$. They have
been largely studied and used after the seminal work of D. Slepian and his co-authors.
In several applications, uniform estimates of the $\psi_{n,c}$ in $n$ and $c,$ are needed.  To progress in this
direction, we push forward  the uniform  approximation error bounds and give an explicit approximation of their values at $1$ in terms of the
 Legendre complete elliptic  integral of the first kind. Also, we give an explicit formula for the accurate  approximation  the eigenvalues of the  Sturm-Liouville operator associated with the PSWFs.\\


\section{Introduction}

Prolate spheroidal wave functions (PSWFs) have been introduced in the sixties by D. Slepian, H. Landau  and  H. Pollak
\cite{Landau2, LW, Slepian1, Slepian3, Slepian2} as a fundamental tool in signal processing. One can  now refer to  \cite{Logan} for their properties, starting from the seminal work of Slepian, Landau and Pollak.
For a fixed value $c>0$, called the bandwidth, PSWFs constitute an orthonormal basis of $L^2([-1, +1]),$ an orthogonal system of $L^2({\bf R})$ and an
orthogonal  basis of  the Paley-Wiener space $B_c,$ given by $B_c=\left\{ f\in L^2({\bf
R}),\,\, \mbox{Support\ \ } \widehat f\subset [-c,c]\right\}.$
Here, $\widehat f$ denotes the Fourier transform of $f.$ One possible definition is given by the fact that they are eigenfunctions of the
compact integral operators
 $\mathcal F_c$ and $\mathcal Q_c= \mathcal F_c^*\mathcal F_c$, defined  on $L^2([-1,1])$ by
\begin{equation}\label{eq1.1}
 \mathcal F_c(f)(x)= \int_{-1}^1
e^{i\, c\, x\, y} f(y)\, dy,\quad \mathcal Q_c(f)(x)=\int_{-1}^1\frac{\sin c(x-y)}{\pi (x-y)}\, f(y)\,
dy. \end{equation}
On the other hand, Slepian and Pollack have pointed out, and reported this property as   "a lucky incident",  that  the
operator $\mathcal Q_c$ commutes with a Sturm-Liouville operator $\mathcal L_c$, which is first  defined on $C^2([-1,1])$ by
\begin{equation}\label{eq1.0}
\mathcal L_c(\psi)=-\frac{d}{d\, x}\left[(1-x^2)\frac{d\psi}{d\,x}\right]+c^2 x^2\psi.
\end{equation}
So PSWFs $(\psi_{n,c})_{n\geq 0}$ are also eigenfunctions of $\mathcal L_c$. They are ordered in such a way that the corresponding eigenvalues of $\mathcal L_c$, called $\chi_n(c)$, are strictly increasing. Functions $\psi_{n,c}$ are restrictions to the interval $[-1, +1]$ of real analytic functions on the whole real line and eigenvalues $\chi_n(c)$ are the values of  $\lambda$ for which the equation $\mathcal L_c \psi=\lambda \psi$ has a bounded solution on the whole interval.\\

We will use all along this paper the fact that the PSWFs are eigenfunctions of $\mathcal L_c$. Remark that the study of PSWFs as eigenfunctions of the above Sturm-Liouville problem has started a long time ago. To the best of our knowledge,  C. Niven was the first, in 1880, to give  a remarkably detailed theoretical, as well as computational studies of the eigenfunctions and the eigenvalues of $\mathcal L_c,$  see \cite{Niven}. Nowadays work on PSWFs  is mainly connected with possible applications in signal processing \cite{Boyd2, Goss,  Li, Lin} and other scientific issues.  In geophysics, for instance,  they provide good approximations of the Rossby waves that constitute the planetary scale waves in the atmosphere and ocean, see \cite{Dickinson, Miles2, Muller1, Muller2, Oconnor}. \\

One main issue concerns numerical computations of the PSWFs and related quantities,\cite{Boyd1,  Karoui1, Xiao, Walter}. This is a necessary step before using the bases they constitute.
For $c=0$,  PSWFs reduce to Legendre polynomials, which are extensively used to expand functions on $[-1, +1]$. Nevertheless numerical evidence shows that in many cases PSWFs may be more adapted, beyond band limited signals. Accurate estimates are then needed  to fix which   bandwith $c$, and so which  specific basis,  will be used to decompose the signal. \\

Approximation of PSWFs in terms of the Bessel function $J_0$ (or other special functions when considering different generalizations of the PSWFs) has been developed by many authors (see for instance \cite{Dunster}). It is based on WKB approximation of the PSWF as eigenfunctions of the Sturm-Liouville problem that we described above.
 The case of (normalized) Legendre polynomials has been studied for a long time and may be summarized in the  formula
$$ \overline P_n(\cos \theta)\approx (n+1/2)^{1/2} \left(\frac{\theta}{\sin \theta}\right)^{1/2}J_0\left((n+1/2) \theta\right),$$
for which very precise bounds of the approximation error are known (see for instance \cite{Gatteschi}). Such formulas are sometimes called Hilb's formulas. When  Legendre polynomials are replaced by $\psi_{n, c}$,  the quantity $n+1/2$ is partly replaced by $\sqrt{\chi_n(c)}$ or partly replaced by $\psi_{n, c}(1)^2$, since there is no simple relation between both.  The change of variable linked to $\cos \theta$ is replaced by an expression that depends on $q=c^2/\chi_n(c)$ and involves Legendre elliptic integrals. In this work, we mainly restrict ourselves to the values of $n, c$ such that $q=c^2/\chi_n(c)<1$. Condition $q\leq 1$ guarantees that  $\psi_{n, c}$ oscillates on the whole interval $(-1, +1)$, like Legendre polynomials, and gets its largest value at $1$. \\

We briefly  describe  some features of our study. In a first step the equation is transformed into its Liouville normal form, so that one can make use of Olver's theorems to have precise error bounds when approximating the function $\psi_{n, c}$ by a quantity that involves $\chi_n(c)$ and $\psi_{n,c}(1)$. It should be emphasized that  this formula does not appear here for the first time. Moreover the fact that one has uniform estimates when $q$ stays far from $1$ may be found in \cite{Dunster}. But we push forward estimates, to allow $q$ to tend to $1$ and have an explicit error of
order $[(1-q)\sqrt{\chi_n(c)}]^{-1}$.

Our second step consists in having an accurate  estimate  for $\psi_{n, c}(1)$, which appears as a coefficient to be fixed
under the condition that the function $\psi_{n,c}$ has a unit $L^2([-1,1])$-norm. We find the explicit approximate formula,
$$\psi_{n, c}(1)\approx\chi_n(c)^{1/4} \sqrt{\frac{\pi}{2 \mathbf  K(\sqrt{q})}}$$
where $\K$ is the complete Legendre elliptic integral. The relative error estimate is found of the same order $[(1-q)\sqrt{\chi_n(c)}]^{-1}$.

It should be mentioned that the above estimate of $\psi_{n, c}(1)$ plays a central role in our further study of the
sharp decay of the spectrum of the operator $\mathcal Q_c$, see \cite{Bonami-Karoui2}.

 \smallskip

 It turns out that  our approximate expressions depend only  on the values of the eigenvalues $\chi_n(c).$
 The explicit approximation of the $\chi_n(c)$ is also  one of our concerns. There exist accurate numerical methods for computing the $\chi_n(c),$
 but theoretical studies have their own interest.   For instance the condition $q<1$  has been  discussed by Osipov in \cite{Osipov} in view of  replacing it by a condition  that involves only $n$ and $c$, not $\chi_n(c)$. He proved  that $q<1$ when $c<\frac{\pi n}2$, while $q>1$ when $c> \frac{\pi (n+1)}2$. His study extends largely this particular comparison  of $q$ with $1$. We further  extend  this study  and state the final result in a more friendly way: we prove that there exists an explicit function $\Phi$, which may be written in terms of elliptic integrals, such that
\begin{equation}\label{intro}
  \Phi \left(  \frac {2c}{\pi(n+1)}\right) <\sqrt q < \Phi \left(\frac {2c}{\pi n}\right).
  \end{equation}
Moreover, we prove that, for $q<1$,
$$\sqrt{q}\approx \Phi \left(  \frac {2c}{\pi(n+1/2)}\right)$$
with an error estimate of $O(1/n^2).$ Formula \eqref{intro} is also used to interpret conditions of the type $(1-q)\sqrt{\chi_n(c)}>\kappa$, which appear everywhere, in terms of $n$ and $c$.  Roughly speaking, this is comparable with  $n-\frac{2c}\pi>\kappa'\log c$. This condition is reminiscent of the description by Landau and Widom \cite{LW} of the decay of the eigenvalues of $\mathcal Q_c$: in the above range of $n,$ the decay is super-exponential.

\smallskip

Let us emphasize the fact that we make a special effort towards numerical constants, in order to show that most of them remain small. Numerical experiments will be presented elsewhere.

\smallskip

This work is organized as follows. Section 2 is centred on the eigenvalues $\chi_n(c)$ and formula \eqref{intro}.
 Section 3 is the main section, with
accurate uniform estimates of the PSWFs for $q<1$. They are given in terms of $\chi_n(c)$ and $\psi_{n,c}(1)$ at first, then only in terms of  $\chi_n(c)$ later on. As a consequence we give an improved  error bound for the approximation
of $\chi_n(c)$ and hence of the quantity $q$. In Section 4,  we first extend the previous techniques to get also uniform estimates of  $\psi_{n,c}$
when $q=1$. Then the same kind of estimates is  used to approximate $\psi_{n, c}$ by the Legendre polynomial $P_n$.

\smallskip
Without loss of generality we assume everywhere that $\psi_{n,c}(1)>0$. To alleviate notation, we systematically replace
  $\psi_{n,c}$ by $\psi_n$ and $\chi_n(c)$ by $\chi_n,$  the parameter $c$ being implicit.

\section{Bounds and estimates of the eigenvalues $\chi_n.$}

Inequalities \eqref{intro}, which complete previous work of Osipov, are based on properties of the equation satisfied by the  PSWFs. We start by a study of this equation.
For simplification, we  skip the parameter $c$  and note  $\chi_n, \psi_n$. We also skip the parameter $n$ for $q=q_n(c)=c^2/\chi_n(c).$ The equation satisfied by $\psi_n$ is then given by
\begin{equation}\label{eqq2.1}
 \frac{\mbox{d}}{\mbox{d} x}\left[(1-x^2)\psi_n'(x)\right]+\chi_n (1-qx^2) \psi_n(x)=0,\quad x\in [-1,1].
 \end{equation}
Recall that the function $\psi_n$ is smooth up to the boundary.  Its $L^2$ norm on $[-1, 1]$ is equal to $1$ and  $\psi_n(1)>0$.

 Because of the parity of the PSWFs ($\psi_n$ has the same parity as $n$), we can restrict to the interval $[0, 1]$, which we do now.
 We then use the  Liouville transformation, which transforms the equation \eqref{eqq2.1} into its Liouville normal form
 \begin{equation}\label{eqq2.3}
 U''+ \left( \chi_n + \theta\right) U= 0.
 \end{equation}
 In view of this, we first let
 \begin{equation}\label{eeqq2.1}
 S(x)=S_{ \sqrt{q}}(x)=\int_x^{\min (1, \frac 1{\sqrt q})} \sqrt{\frac{1-qt^2}{1-t^2}}\, dt.
 \end{equation}
 The function $S$ can be written as
$$S(x) =\mathbf E(\sqrt{q})- \mathbf E(x,\sqrt{q}),$$ where    $${\displaystyle \mathbf E(k)=\int_0^{\min (1, \frac 1{k})} \sqrt{\frac{1- k^2 t^2}{1-t^2}}\, dt,\,\,\qquad \mathbf E(x,k)=\int_0^x \sqrt{\frac{1- k^2 t^2}{1-t^2}}\, dt}.$$
When $0\leq k\leq 1,$ we recognize from the previous equalities, the complete and incomplete elliptic integral of the second kind, respectively.
Note that $S(\cdot)$ defines a homeomorphism on the whole interval  $[0, \min (1, \frac 1{\sqrt q})]$.   Liouville Transformation consists in looking  for $\psi$ under the form
$$
 \psi_n(x)= \varphi(x) U(S(x)),\quad \varphi(x)= (1-x^2)^{-1/4}(1-q x^2)^{-1/4}.
$$
The equation satisfied by $U$  may be written as in \eqref{eqq2.3}
with $$\theta (S(x))= \varphi(x)^{-1} (1-qx^2)^{-1}\frac{\mbox{d}}{\mbox{d} x}\left[(1-x^2)\varphi'(x)\right].$$
We have $$\varphi'/\varphi=-\frac 14 Q'/ Q, \qquad \qquad Q(x)=(1-x^2)(1-qx^2).$$
 It follows that $\theta \circ S$ is a rational function with poles at $\pm 1$ and $\pm \sqrt {\frac 1q}$, which may be written
   $$\theta \circ S= \frac 1{16}(1-qx^2)^{-1} \left[(1-x^2)\left(\frac{Q'}{Q}\right)^2 -4\frac{\mbox{d}}{\mbox{d} x}\left((1-x^2)
   \frac{Q'}{Q}\right)\right].$$
By computing the different derivatives appearing in the previous expression, then by writing the numerator as a polynomial in $1-x^2$, one
can easily check that
\begin{eqnarray}\label{h1S}
\theta (S(x))&=&  \frac{(1-q)^2}{4(1-x^2)(1-q x^2)^3}+\frac{(1-q)^2+2q(3-q)(1-x^2)}{4(1-q x^2)^3}.
\end{eqnarray}
 The following proposition shows the monotonicity of  $\theta\circ S$ for any $q>0.$
  \begin{proposition}\label{increasing}
 For $q>0,$ the function $\theta\circ S$ is increasing on  $[0,\min(1,1/\sqrt{q})).$
 \end{proposition}
\begin{proof}
We recall the expression of
$\theta\circ S$ given in \eqref{h1S}. We use the notation $u=1-x^2$.
Straightforward computations show that
$$ \theta(S(x))'= 2x \frac{G(u)}{4 u^2 (1-q + q u )^4},\quad G(u) = (1-q)^2(1-q+4qu +3qu^2) +2q(3-q)(2qu-1+q)u^2.$$
Hence, it suffices to prove that $G(u)$ is non negative for  $q>0$ and $u$ such that  $\max(0, 1- (1/q))< u \leq 1.$  If $1\leq q \leq 3,$  we deduce from the inequality $1-q +qu \geq 0$ that both terms are non negative. Assuming now that $q<1,$ by computing the minimum of $u(2qu-1+q)$ we get the inequality $8qu(2qu-1+q)\geq -(1-q)^2$.
Substituting the right hand side of the previous inequality in the second term of  $G$, one gets
\begin{eqnarray*}
G(u) &\geq & \frac{(1-q)^2}{4} \left(4-4q +16 q u +12 q u^2 -(3-q)u\right)\\
&\geq & \frac{(1-q)^2}{4} \left(4-4q + u (17 q -3)\right)\geq 0,\quad \forall\,  0<q<1,\quad 0\leq u \leq 1.
\end{eqnarray*}
Finally, assuming that $q>3$, direct computations show that
\begin{eqnarray*}
G(u) &\geq & (1-q)^2 (3q -3 +3 q u^2)-2 q (q-3)(2qu+q-1)u^2\\
&\geq & 3 (q-1)^3 +\left( 3q (1-q)^2 -2q(q-3) (3q-1)\right) u^2\\
&\geq & \left (6(q-1)^3-2q(q-3)(3q-1)\right ) u^2
\end{eqnarray*}
which is positive.
\end{proof}

Let us go back to   the eigenvalues $\chi_n$. They satisfy the classical inequalities (the left hand side has been slightly improved in \cite{Bonami-Karoui1} but we do not use this)
\begin{equation}
\label{bounds1-chi}
n (n+1)   \leq \chi_n \leq n(n+1)+c^2.
\end{equation}
On the other  hand, it has been shown in [Theorem 13, \cite{Osipov}], that if $n\geq 2$ and $c^2/\chi_n <1,$ then
\begin{equation}
\label{bounds3-chi}
\chi_n < \left(\frac{\pi}{2} (n+1)\right)^2 .
\end{equation}

As an application of the previous proposition we give new inequalities for  $\chi_n$, which improve or complete  the above bounds,
and are valid for $q\leq 1$ as well as for $q>1.$ We will use the fact that $\psi_n$ has exactly $n$ zeros in $(-1, +1)$.
Instead of  the change of variables $S$, we define $\widetilde S$  on $(-\min(1,1/\sqrt{q}),\min(1,1/\sqrt{q}))$ by
$$
  \widetilde S(x) = \left\{\begin{array}{rl}
  \E(x,\sqrt{q}) \ \ & \mbox{for} \ x\geq 0 \\
         -\E(-x,\sqrt{q}) \ \ & \mbox{for} \ x\leq 0
                      \end{array}\right..
$$
It is easily seen that the function $\widetilde U$, which is such that $ \psi_n(x)= \varphi(x) \widetilde U(\widetilde S(x))$, satisfies the equation
\begin{equation}\label{tilde}
     Y''+ \left( \chi_n + \widetilde \theta\right) Y= 0,
\end{equation}
with $\widetilde \theta$ an even function such that for $x>0$, we have $\widetilde \theta (\widetilde S(x))=\theta(S(x))$. We know from Proposition \ref{increasing} that $\widetilde \theta\geq \chi_n+\frac {q+1}2$ on  the interval $(-\mathbf E(\sqrt{q}), +\mathbf E(\sqrt{q}))$.
 We use Sturm comparison theorem between the equation \eqref{tilde} and the equation $Y''+\left(\chi_n+\frac {q+1}2\right) Y=0$.  This allows us to say that the distance between two consecutive  zeros of the equation \eqref{tilde} is smaller than $\frac{\pi}{\left(\chi_n+\frac {q+1}2\right)^{\frac 12}}$. On the other hand, we know that $\widetilde U$, whose zeros correspond to the ones of $\psi_n$, has exactly $n$ zeros in $(-\min(1,1/\sqrt{q}),\min(1,1/\sqrt{q}))$ (see \cite{Osipov} for $q>1$). As a consequence, we find that
\begin{equation}\label{chi+}
\frac 2\pi \mathbf E(\sqrt q) \sqrt{\chi_n+\frac{q+1}2}\leq n+1.
\end{equation}
Let $\Phi$ denotes the inverse function of the function $k\mapsto \frac k{\mathbf E(k)}$, so that $\Phi\left( \frac k{\mathbf E(k)}\right)=k$. It is an increasing function that  vanishes at $0$ and takes the value $1$ at $1$. Then we have the following theorem, which gives a double inequality for $\sqrt{q}$ and implies also a double inequality for $\chi_n$.
\begin{theorem} \label{chi-between2}
For all $c>0$ and $n\geq 2$ we have
\begin{equation}\label{ineqPhi}
\Phi \left(  \frac {2c}{\pi(n+1)}\right) < \frac c{\sqrt{\chi_n}} < \Phi \left(\frac {2c}{\pi n}\right),
\end{equation}
where $\Phi$ is the  inverse of the function $k\mapsto \frac k{\mathbf E(k)}.$
\end{theorem}
\begin{proof}
The left hand side comes directly from \eqref{chi+} and the monotonicity of $\Phi$, while the right hand side is a consequence of Proposition 3 in \cite{Osipov}.
\end{proof}
We could as well have written a double inequality for $\sqrt {\chi_n}$,
$$
\frac{c}{\Phi\left(\frac{2c}{\pi n}\right)} < \sqrt{\chi_n }<
\frac{c}{  \Phi\left(\frac{2c}{\pi (n+1)}\right)},\quad c>0,\quad n\geq 2.
$$
It may be rewritten as
\begin{equation}
\label{boundschi2}
c\, \widetilde\Phi\left(\frac{\pi n}{2c}\right) < \sqrt{\chi_n }< c\, \widetilde\Phi\left(\frac{\pi (n+1)}{2c}\right),
\end{equation}
with $\widetilde \Phi(k)=[\Phi(\frac 1k)]^{-1}$. This function is the inverse of the function $k\mapsto k\E(\frac 1k)$ whose derivative, for
any real $k>1$, is given by
$$\E\left(\frac 1k\right)+\int_0^1\frac{t^2}{\sqrt{(1-t^2)(1-(t/k)^2)}}\, dt=\E\left(\frac 1k\right)+\K\left(\frac 1k\right)-\int_0^1\sqrt{\frac{1-t^2}{1-(t/k)^2}}\, dt.$$
Note that this last term is bounded below by $\K\left(\frac 1k\right)$, which in turn is bounded below by  $\frac \pi 2.$
So the derivative of $\widetilde \Phi$ is bounded by $\frac 2\pi $ and
$$c\, \widetilde\Phi\left(\frac{\pi (n+1)}{2c}\right)-c\, \widetilde\Phi\left(\frac{\pi n}{2c}\right) < 1.$$
It is natural to choose the middle value for an approximate value of $\sqrt{\chi_n},$ that is,
\begin{equation}
\label{approxchi}
\sqrt{\widetilde \chi_n} =  \frac{c}{\Phi\left(\frac{2c}{\pi (n+1/2)}\right)},\quad n\geq \frac{2c}{\pi}.
\end{equation}
We define also $\widetilde q=\frac{c^2}{\widetilde\chi_n}.$
It is easy to check that we have the following approximation and relative approximation errors of $\sqrt{\chi_n}:$
\begin{equation}\label{error-chi}
 \left|\sqrt{\chi_n}-\sqrt{\widetilde \chi_n}\right|\leq \frac{1}{2},\qquad
  \left|\frac{\sqrt{\chi_n}-\sqrt{\widetilde \chi_n}}{\sqrt{\chi_n}}\right|\leq \frac{1}{2\sqrt{\chi_n}}\leq \frac{1}{2n}.
\end{equation}
As a consequence, we also have
\begin{equation}\label{error-q}
| \sqrt q- \sqrt{\widetilde q}| \leq  \frac{c}{2\sqrt{\chi_n}\sqrt{\widetilde \chi_n}}\leq\frac{c}{n(2n+1)}.
\end{equation}
\begin{remark} Formula \eqref{approxchi} provides us with an approximation $\widetilde {\chi}_n$ of $\chi_n$ in terms of the easily computed
function $\Phi.$ This approximation may be compared with classical numerical methods to compute $\chi_n$, such as
Flammer's method, see \cite{Flammer}.  Numerical experiments prove that it is a good approximation for $n$ not too small. We will see later on that the relative error for $\sqrt{\chi_n}$ is of order $O(1/n^2)$ for $q$ not very close to  $1.$ For large values of $n,$ our  formula
\eqref{approxchi} provides us with precise values of $\sqrt{\chi_n}$ with very low computational load compared to the classical methods.
\end{remark}

Theorem \ref{chi-between2} provides also new upper and lower bounds of the $\chi_n$ which are valid for ${\displaystyle n< \frac{2c}{\pi}-1.}$
 This is the subject of the following proposition. The  left hand side of the inequality \eqref{encadrement} below
 has already been stated and proved in \cite{Bonami-Karoui1}.

\begin{proposition} For  $n\geq 2$ and
 $c>\frac {\pi (n+1)}2,$  we have the inequalities
 \begin{equation}\label{encadrement}
\frac{\pi cn}2\leq\chi_n\leq 2c(n+1).
\end{equation}
\end{proposition}

\begin{proof}
We first prove that for $s>1$,
\begin{equation}\label{Phiforq>1}
   \sqrt{ \frac {\pi s}{4}}<\Phi(s)< \sqrt{s}.
\end{equation}
Recall that the inequality $s>1$ is equivalent to $k=\Phi(s)>1$. So it is sufficient to prove that
$$\sqrt{\frac{\pi k}{4\E(k)}}<k< \sqrt{\frac{k}{\E(k)}},$$
that is,
\begin{equation}\label{Eforq>1}
    \frac {\pi}{4k}<\mathbf E(k)< \frac {1}{k},\quad k>1.
\end{equation}
After a change of variables we have $k\E(k)=\int_0^1\sqrt{\frac{1-t^2}{1-(t/k)^2}}dt.$ This latter decreases from $1$ to  $\int_0^{1}\sqrt{1-t^2} dt =\frac {\pi}{4}$. We have proved \eqref{Eforq>1} and \eqref{Phiforq>1}. The rest of the proof of \eqref{encadrement} is a straightforward consequence of Theorem \ref{chi-between2}, using the fact that $\frac{2c}{\pi n}>\frac{2c}{\pi (n+1)}>1$.
\end{proof}

\medskip

Before finishing this section, we use Theorem \ref{chi-between2} to give bounds for the quantity $(1-q)\sqrt{\chi_n}.$ This  allows us  to interpret conditions on $(1-q)\sqrt{\chi_n}$
 in terms of $c$ and $n$. The  following proposition says, roughly speaking, that one has to add a factor of $\log n$ when passing from a condition on $\chi_n$ to a condition on $n$.

\begin{proposition}\label{comparison}
For $n\geq 2$ and $q<1,$ we have the inequalities
\begin{equation}\label{NtoCh}
(1-q)\sqrt{\chi_n}\geq \frac {(n-\frac{2c}{\pi})-e^{-1}}{\log n +5},
\end{equation}
\begin{equation}\label{ChtoN}
n+1-\frac{2c}{\pi}\geq \frac 1{\pi^2}(1-q)\sqrt{\chi_n}\log \left(\frac 1{1-\sqrt{q}}\right).
\end{equation}
\end{proposition}
\begin{proof}
The proof will make use of the  complete Legendre elliptic integral of the first kind, which we denote by $\mathbf K$. Recall that \begin{equation}\label{K}
\mathbf K(\eta)=\int_0^1 \frac{ dt}{\sqrt{(1-t^2)(1-\eta^2 t^2)}},\quad 0<\eta <1.
\end{equation}
We will need precise estimates on the behavior of $\mathbf K$, namely
\begin{equation}\label{behaviorK}
(1-\eta)\frac{\pi}{2}+ \frac 12 \log \frac{1+\eta}{1-\eta}\leq \mathbf{K}(\eta)\leq \frac \pi 2 +\frac 12 \log \frac{1+\eta}{1-\eta}.
 \end{equation}
 To prove this, we  take the difference between $\mathbf{K}(\eta)$ and the integral of
  $\eta/(1-\eta^2 t^2)$.

Let us go back to the proof of the proposition. We write $\frac c{\sqrt{\chi_n}}= 1-\frac{\delta}{n}$ and $\frac{2c}{\pi n}=1-\frac{\delta^*}{n}$.  Using Theorem 1, we have
$$ \Psi\left(1-\frac{\delta}{n}\right)\leq 1-\frac{\delta^*}{n}$$
with $\Psi(k)=\frac{k}{\E(k)}$ the inverse function of $\Phi$. By using the fact that $\E(\cdot)$ is decreasing, one gets
$$\frac{\delta^*-\delta}{n}\leq \E(1-\frac{\delta}{n})-1.$$
We need to estimate of $\E(\cdot)-1$. Writing this quantity as an integral, we get  bounds in terms of elliptic integral  $\K$, given by
\begin{equation}
(1-k^2)(\K(k)/2-1)\leq \E(k)-1\leq (1-k^2)\K(k).
\end{equation}
By using \eqref{behaviorK}, one obtain the
 inequalities
\begin{equation}\label{terminal}
 ({1-k}) \left(\frac{1}{4}\log \frac{1}{1-k}-1\right)\leq \E(k)-1\leq  (1-k)\left(\pi+ \log \frac{2}{1-k} \right),\quad 0\leq k<1.
 \end{equation}
Using this last inequality, we get
$$\delta^*\leq \delta \left(\pi +1+ \log \frac{2n}{\delta}\right).$$
Since $\delta \log(\frac 1\delta)\leq e^{-1}$, we finally find the inequality
$$\delta^*-e^{-1} \leq \delta \left(\pi +1+ \log {2} +\log n\right)\leq \delta \left(\log n +5\right),$$
from which we conclude at once.

Conversely, we change slightly the notation and write
$\frac c{\sqrt{\chi_n}}= 1-\frac{\delta}{(n+1)}$ and $\frac{2c}{\pi (n+1)}=1-\frac{\delta^*}{(n+1)}$.
 Using Theorem 1 and the fact that $\mathbf E(\cdot)$ is bounded by $\pi/2$, we get the inequality
$$\E\left(1-\frac{\delta}{n+1}\right)-1\leq \frac \pi 2 \frac{\delta^*}{n+1}.$$
 If we use the first inequality in \eqref{terminal}, we get the following.
 $$\frac{\delta^*}{n+1}\geq \frac{2}{\pi}+\frac{2}{\pi}\frac{\delta}{n+1}\left( \frac{1}{4}\log\frac{1}{1-\sqrt{q}}-1\right).$$
 We obtain \eqref{ChtoN} by using the inequality $\sqrt{\chi_n}\leq \frac{\pi}{2}(n+1)$.
\end{proof}
\medskip

\section{WKB  approximation of the PSWFs  and corollaries.}

 We assume in this section that $q=c^2/\chi_n<1$. We  first give explicit uniform approximation for  the PSWF $\psi_n$  in terms of its value at $1,$
 as well as in terms  of the Bessel function $J_0$ and of the associated eigenvalue $\chi_n.$ This approximation holds
 under the condition that  $(1-q) \sqrt{\chi_n}$ is large enough. We rely on properties of Sturm-Liouville equations and use the estimates given by   Olver in his book \cite{Olver}.
  The existence of such an asymptotic approximation  is well-known and has been developed in a larger context, see for example
  \cite{Dunster, Miles1, Olver}. In particular,   asymptotic approximation of the PSWFs has been given in \cite{Miles1}
  for large values of the parameter $c$ while $n-$th order  uniform asymptotic
  approximations of the PSWFs are obtained in \cite{Dunster} as a consequence of  Olver's results. In this paragraph,  we
  go back to Olver's asymptotic approximation scheme and we give precise estimates and simple bounds of the  functions that are involved in the perturbation term.
 As a consequence, we obtain a simple and  practical expression of the approximation error of the $\psi_n.$
Once this done, we  get rid of the dependence in $\psi_{n}(1)$ of these approximations. More precisely we give an approximation of $\psi_{n}(1)$ in terms of $\sqrt{\chi_n}$, $n$ and $c$. Recall that $\psi_{n}$ has $L^2([-1,1])-$norm $1$, which fixes the value $\psi_{n}(1)$.
  This approximation of $\psi_{n}(1)$  is in particular a critical issue for the sharp decay rate of   eigenvalues of the integral operator $\mathcal Q_c$ defined in \eqref{eq1.1}.

  We also use the uniform approximation to   improve the  error
 bounds for our  approximation of the quantities $\chi_n$ and $q$ via Formula \eqref{approxchi}.

 \subsection{Uniform approximation of the PSWFs knowing their value at $1$.}

 Let us go back to the transformed equation \eqref{eqq2.3}, that is,
$$U''(s)+ \left( \chi_n + \theta(s)\right) U(s)= 0,\quad s\in [0, S(0)].$$ We use the notations of the previous section.
We claim that the function $F(\cdot)=F_q(\cdot),$ given by
\begin{equation}\label{FS}
 F(S(x))= \frac{1}{4 S^2(x)}-\theta (S(x)),\quad x\in [0,1),
\end{equation}
 is continuous
 on $[0, S(0)]$. We postpone the proof to Lemma \ref{bounds3} and go on. The equation
 \begin{equation}\label{eqq2.5}
U''(s)+\left(\chi_n +\frac{1}{4 s^2}\right) U(s) = F(s) U(s),\quad s\in [0, S(0)].
\end{equation}
is a particular case of the equation considered in Olver's book, Chapter 12, Theorem 6.1. The associated homogeneous equation  has the
two independent solutions
$$U_1(s)= \chi_n^{1/4}\sqrt{s} J_0(\sqrt{\chi_n} s),\quad U_2(s)=\chi_n^{1/4}\sqrt{s} Y_0(\sqrt{\chi_n} s),$$
where $J_0$ (resp. $ Y_0$) denotes the Bessel function of the first (resp. second)
type. Hence, using  the well known explicit value of the Wronskian of $J_0,\, Y_0,$ the solution $U$ may be written as
\begin{eqnarray}\label{eqq2.6}\nonumber
U(s)&=& A U_1(s)+A'  U_2(s)+\frac{\pi}{2\sqrt{\chi_n}}\\
&&\times \int_0^s \sqrt{st\chi_n}\left[J_0(\sqrt{\chi_n} s)Y_0(\sqrt{\chi_n} t)
-J_0(\sqrt{\chi_n} t)Y_0(\sqrt{\chi_n} s)\right] F(t) U(t)dt.
\end{eqnarray}

From now on, $U$ is the particular solution of \eqref{eqq2.6}  on $[0, S(0)]$ that we have  defined in Section 2, that is,
\begin{equation}\label{defU}
U(S(x))=\left((1-x^2)(1-qx^2)\right)^{1/4}\psi_{n}(x).
\end{equation}
In the next lemma, we  prove that $S(x)/\left((1-x^2)(1-qx^2)\right)^{1/4}$ goes to $1$ as  $x$ goes to $1$, so that
$$
 \lim_{s\rightarrow 0}\frac{U(s)}{\sqrt{s}}=\psi_n(1).
$$
Let us prove that this behavior at $0$ forces the coefficient $A'$ to be $0$. Since the function $J_0(s)$ has the limit $1$, while  $Y_0(s)$ has a singularity at $0$, it is sufficient to prove that the last term in \eqref{eqq2.6} is bounded by $s$, up to a constant.  But  the function inside the integral
in \eqref{eqq2.6}  is bounded. Indeed, we have the classical inequality 
\begin{equation}\label{Ineq10}
\sup_{s\geq 0}  s (J_0^2(s) + Y_0^2(s)) \leq \frac 2\pi.
\end{equation}
(see \cite{Watson}, p. 446-447 for instance). So we have not only proved that $A'=0$ but also that
\begin{equation}
 \label{constantA}
 A=\psi_{n}(1) \chi_n^{-1/4}.
 \end{equation}
We may use  either  $A$ or $\psi_n(1)$ in the next formulas, which is obviously equivalent. 

 Hence,  \eqref{eqq2.6} can be rewritten as
\begin{equation}\label{Udef}
 U(s)=\psi_n(1)\sqrt{s} J_0 (\sqrt{\chi_n}s)+ \frac{1}{\sqrt{\chi_n}} \int_0^s K_n(s,t) F(t)U(t) dt
\end{equation}
with
\begin{equation}
K_n(s, t)=\frac{\pi}{2}\sqrt{st\chi_n}\left[J_0(\sqrt{\chi_n} s)Y_0(\sqrt{\chi_n} t)
-J_0(\sqrt{\chi_n} t)Y_0(\sqrt{\chi_n} s)\right] .\label{K_n}
\end{equation}

The solution $U$ given by \eqref{Udef} is, up to the multiplicative constant $\psi_n(1)$,  the solution of the equation \eqref{eqq2.5} that has been considered by Olver. We first give estimates on  $S(x)$ and $\theta(S(x))$ before we use Olver's inequalities. The two corresponding lemmas contain in particular the properties that we have already used.

\begin{lemma} \label{difS} For $0\leq q <1$ and $x\in [0,1],$ we have
\begin{equation}\label{Ineq4}
-\frac{q(1-x^2)^{3/2}}{2(1-qx^2)^{1/2}}\leq S(x)-\sqrt{(1-x^2)(1-q x^2)} \leq \frac {2-q}3 (1-x^2)^{3/2}.
\end{equation}
Moreover, for $x$ tending to $1$ we can write
\begin{equation}\label{DL}
\sqrt{(1-x^2)(1-q x^2)}/S(x)=1+(\frac q{1-q}+\frac 34)(1-x)+o(1-x).
\end{equation}
\end{lemma}

\begin{proof}
 By computing the derivative of the quantity ${\displaystyle S(x)- \sqrt{(1-x^2)(1-qx^2)}},$ one gets
\begin{equation}\label{Ineq44}
S(x)-\sqrt{(1-x^2)(1-qx^2)}=\int_x^1(1-t)\sqrt{\frac{1-qt^2}{1-t^2}}\, dt - q \int_x^1 t \sqrt{\frac{1-t^2}{1-q t^2}}\, dt.
\end{equation}
On one hand, using the fact that $\sqrt{\frac{1-qt^2}{1+t}}\leq 1 $, we have
\begin{equation}\label{Ineq5}
\int_x^1(1-t)\sqrt{\frac{1-qt^2}{1-t^2}}\, dt\leq \frac{2}{3} (1-x^2)^{3/2}.
\end{equation}
On the other hand, we have
\begin{equation}\label{Ineq55}
\frac{(1-x^2)^{3/2}}3\leq\int_x^1 t\sqrt{\frac{1-t^2}{1-q t^2}}\, dt\leq  \sqrt{\frac{{(1-x^2)}}{{1-q x^2}}}\int_x^1 t dt =\frac{(1-x^2)^{3/2}}{2\sqrt{1-q x^2}}.
\end{equation}
Finally, by combining (\ref{Ineq44}), (\ref{Ineq5}) and (\ref{Ineq55}), one gets (\ref{Ineq4}).
The computation of \eqref{DL} is elementary.
\end{proof}
As a consequence, we get the following  double inequalities
\begin{equation}\label{equS}
(1- \frac q2) \sqrt{(1-x^2)(1-q x^2)} \leq S(x) \leq \frac{5-q}3\sqrt{(1-x^2)(1-q x^2)}.
 \end{equation}

The following lemma concerns the function $F$, which has been  defined by \eqref{FS}.
 \begin{lemma}\label{bounds3}
For $0\leq q < 1$  the function $F$ is continuous on $[0, S(0)]$. Moreover, we have
\begin{equation}\label{Bounds3}
\left| F(S(x))\right|\leq \frac{3+2q}{4}\frac{1}{(1-qx^2)^{2}},\quad x\in [0,1]
\end{equation}
 and
\begin{eqnarray}\label{L1F}
\alpha_q = (1-q)\int_0^{S(0)}\left| F(s)\right|ds &\leq & \frac{3+2q}{4}\E(\sqrt q)\\
&\leq& 1.5.\label{L1F2}
\end{eqnarray}

\end{lemma}
\begin{proof} The function $F$ is a priori only defined on $[0, S(0))$ but we will prove that we can extend it  at $1$ by continuity. We use the notation $F$ for its extension as well.   We first consider
\begin{eqnarray*}
\theta (S(x))-\frac{1}{4(1-x^2)(1-q x^2)}&=&\frac{(1-q)^2-(1-qx^2)^2}{4(1-x^2)(1-q x^2)^3}+\frac{(1-q)^2+2q(3-q)(1-x^2)}{4(1-q x^2)^3}\\
&=& \frac{1+q(2+3 x^2 q -6x^2)}{4(1-q x^2)^3},\quad 0\leq x< 1.
\end{eqnarray*}
This extends to a continuous function on $[0,1].$ Moreover, from the elementary inequality
 $$ (1-3q x^2)(1-q x^2)\leq 1+q(2+3 x^2 q -6x^2) \leq (1+2q)(1-qx^2),$$
 we conclude that
\begin{equation}\label{Ineq6}
\frac{3q-1}{4(1-q x^2)^2}\leq \theta (S(x))-\frac{1}{4(1-x^2)(1-q x^2)}\leq  \frac{1+2q}{4(1-q x^2)^2},\quad x\in [0,1).
\end{equation}
Next, the extension into a continuous function at $1$ of $\frac{1}{(1-x^2)(1-q x^2)}-\frac{1}{S^2(x)}$ at $1$ is an easy consequence of \eqref{DL}. We  then use
 (\ref{equS})  and (\ref{Ineq4}) to  conclude that
\begin{equation}\label{Ineq7}
-\frac{3q}{(1-q x^2)^2}\leq \frac{1}{(1-x^2)(1-q x^2)}-\frac{1}{S^2(x)}\leq \frac{2}{ (1-q x^2)^2},\quad x\in [0,1].
\end{equation}
Finally, by combining  (\ref{Ineq6}) and (\ref{Ineq7}), one gets  (\ref{Bounds3}).

It remains to prove \eqref{L1F}. Using   the estimate on $F$ given by \eqref{Bounds3}, together with the change of variable $t=S(y)$, we get
$$ \int_0^{S(x)} |F(t)|\, dt \leq \frac{3+2q}{4}\int_x^1 \frac{ dy}{(1-y^2)^{1/2} (1-q y^2)^{3/2}}.$$
 Straightforward computations give  the  classical identity
 \begin{equation}\label{Ineq12}
\int_x^1 \frac{ dy}{(1-y^2)^{1/2} (1-q y^2)^{3/2}}= \frac{1}{1-q} \left(\frac{ q x\sqrt{1-x^2}}{\sqrt{1-q x^2}}+ S(x)\right).
\end{equation}
 We conclude directly for  (\ref{L1F}) by using \eqref{Ineq12} with $x=0$.
For the inequality  (\ref{L1F2}), we use the concavity of the function $q\mapsto \E(\sqrt{q})$ as well as  the fact that its derivative is $-\frac \pi 8$ at $0$  and  prove that $\E(\sqrt{q})\leq \frac \pi 2(1-\frac q4)$, from which we conclude.
\end{proof}

Let us go back to the function $U$ defined in \eqref{defU} and write
\begin{equation}
U(s)= \psi_n(1)\sqrt s J_0(\sqrt{\chi_n}s)+ \mathcal E_n(s).
\end{equation}
This gives a WKB approximation of $U$ by the first term. We now give estimates of the approximation error $\mathcal E_n$.
Using Theorem 6.1 of Chapter 12 of  Olver's book \cite{Olver}, one has
\begin{equation}\label{error_bound1}
|\mathcal E_n (s)|\leq \psi_n(1) \sqrt{s}\frac{M_0(\sqrt{\chi_n}s)}{E_{0}(\sqrt{\chi_n}s)} \left[\exp\left(\frac{\pi}{2}\int_0^s t M_{0}^2(ut) |F(t)|\, dt\right)-1\right].
\end{equation}
Here
  $$ E_{0}(x)=\left\{ \begin{array}{ll} (-Y_{0}(x)/ J_{0}(x))^{1/2} &\mbox{ if } 0<x\leq X_{0}\\ 1 &\mbox{ if } x\geq X_{0}
 \end{array} \right.,\quad M_{0}(x)=\left\{ \begin{array}{ll} (2|Y_{0}(x)| J_{0}(x))^{1/2} &\mbox{ if } 0<x\leq X_{0}\\
 (J^2_{0}(x)+Y^2_{0}(x))^{1/2} &\mbox{ if } x\geq X_{0}  \end{array} \right.,$$
with $X_0$ the first zero of $$J_{0}(x)+Y_{0}(x)=0.$$
It follows from the classical inequality \eqref{Ineq10} that
\begin{equation}\label{boundsM0}
t M_0^2( \sqrt{\chi_n}t)\leq \frac{2}{\pi}\frac{1}{\sqrt{\chi_n}},\quad \chi_n^{1/4} \sqrt{s} \frac{M_0(\sqrt{\chi_n} s)}{E_0(\sqrt{\chi_n} s)}\leq \sqrt{\frac{2}{\pi}}.
\end{equation}
Before stating Olver's estimates in the form that we will use later on, let us recall or fix some notations.  \\

\noindent {\bf Notations.}  {\sl For $0\leq q\leq 1$ the constant $\alpha_q$ is given by
  \begin{equation}
  \label{alpha_q}
  \alpha_q =(1-q)\int_0^{\E(\sqrt q)}\left| F(s)\right|ds,
 \end{equation}
 where $F=F_q$ has been defined in \eqref{FS}.\\
 For  a positive integer $n, c$ such that $q=c^2/\chi_n <1,$  the quantity $\varepsilon_n$ is defined as
  \begin{equation}
  \label{notation}
  \varepsilon_n =\frac{1}{(1-q)\sqrt{\chi_n}}.
 \end{equation}
The function $\Theta$ is defined on $(0,+\infty)$ by
\begin{equation}
  \label{notation2}
  \Theta(x)=\frac{e^x-1}{x},\quad x>0.
 \end{equation}
}
 We will be mainly interested in $\Theta$ on the interval $(0, 1)$, where $\Theta (x)\leq 1+x\leq 2$ plays the role of a multiplicative constant.\\

By combining    (\ref{error_bound1}) and (\ref{boundsM0})  one gets the estimates
\begin{equation}\label{error_bound2}
|\mathcal E_n(s)|\leq \sqrt{\frac{2}{\pi}}\frac{\psi_n(1)}{\chi_n^{1/4}} \left(\exp \left((1-q)\varepsilon_n\int_0^s|F(t)|dt\right)-1\right).
\end{equation}
The following statements  provide us with an error bound for the WKB uniform approximation of the function $U$ and consequently of the function $\psi_n$. They are direct corollaries of the previous inequality.

\begin{lemma}\label{errorU} Let $n, c$ be such that $q=c^2/\chi_n <1$. Then
the  function $U$ defined in \eqref{defU} is given by
$$ U(s) =  \psi_n(1) \sqrt{s} J_0(\sqrt{\chi_n}s)+ \mathcal E_n(s),$$
with
\begin{equation}
\sup_{s\in(0, S(0)}|\mathcal E_n(s)|\leq \sqrt{\frac{2}{\pi}}\frac{\psi_n(1)}{\chi_n^{1/4}}\alpha_q \epsilon_n\Theta(\alpha_q \epsilon_n).
\end{equation}
\end{lemma}

\begin{proposition} \label{uniform-prop}
 let  $n, c$ be such that
 $q=c^2/\chi_n <1$.  Then under the previous notations, one can write
\begin{equation}\label{uniform0}
    \psi_{n}(x)=  \psi_{n}(1)\frac{\sqrt{S(x)}J_0(\sqrt{\chi_n} S(x))}
{(1-x^2)^{1/4}(1-q x^2)^{1/4}}+ R_{n}(x)
\end{equation}
for $0\leq x\leq 1$, with
\begin{equation}\label{bounds1}
| R_{n}(x)|\leq  \;\kappa_{0} \frac { (1-x^2)^{1/4}}{(1-qx^2)^{3/4}} \varepsilon_n.
\end{equation}
Here,  $\kappa_{0}=\frac{5}{2}\sqrt{\frac{2}{\pi}}\frac{\psi_n(1)}{\chi_n^{1/4}}\Theta(\alpha_q\varepsilon_n).$
\end{proposition}

\begin{proof}
Lemma \ref{errorU} is directly deduced from \eqref{error_bound2} and the definition of $\Theta$. Let us prove the proposition. Since the function $\Theta$ is increasing,  we have the estimate
$$|\mathcal E_n(s)|\leq \sqrt{\frac{2}{\pi}}\frac{\psi_n(1)}{\chi_n^{1/4}}(1-q)\varepsilon_n\Theta(\alpha_q\varepsilon_n) \int_0^s|F(t)|dt.$$
So, if we use the bound given in \eqref{Ineq12} for the integral, we get the inequalities

$$
|\mathcal E_n(S(x))|  \leq   \frac{2\kappa_0 (1-q)\varepsilon_n}{5}\int_0^{S(x)}|F(t)|dt
\leq  \frac{\kappa_0 \varepsilon_n}{2} \left(\frac{ q x\sqrt{1-x^2}}{\sqrt{1-q x^2}}+ S(x)\right)
 \leq \kappa_0 \varepsilon_n\sqrt{\frac{1-x^2}{1-qx^2}}.
$$
For the last inequality we have first used \eqref{equS}, then the fact that
$qx+S(x)\sqrt{\frac {1-qx^2}{1-x^2}}$ is bounded by $ qx+\frac{2}{\sqrt{1+x}}(1-qx^2),$ then the fact that this last function is bounded by $2$.
We conclude  for \eqref{bounds1} by dividing by $(1-x^2)^{1/4}(1-q x^2)^{1/4}$.
\end{proof}

We end this subsection by the remark that the same kind of estimates, but with larger constants, could have been obtained directly when $\alpha_q \varepsilon_n<1$, without referring to Olver's techniques. Indeed, if we go back to \eqref{Udef} and use  \eqref{Ineq10} to see that the kernel $K_n$ is bounded by $1,$  we have the inequality
\begin{equation}
\label{eqU}
|U(s)|\leq \psi_n(1)\sqrt s |J_0(\sqrt{\chi_n}s)| +\frac{1}{\sqrt{\chi_n}}\int_0^s |F(t)||U(t)| dt\leq \sqrt{\frac{2}{\pi}}\frac{\psi_n(1)}{\chi_n^{1/4}} +\alpha_q \epsilon_n \sup_t|U(t)|.
\end{equation}
This gives a bound above for the maximum of $|U(t)|$ under the assumption that $\alpha_q \epsilon_n<1$, which we can use to estimate the
remainder  term, that is $\mathcal E_n$. We leave the details to the reader.

\subsection{Estimates and bounds of $\psi_{n}(1).$}

As we have already mentioned, the  estimate of  $\psi_{n}(1)$, under the adopted normalization $\|\psi_{n,c}\|_{L^2([-1,1])}=1$, is  a main issue. At this point,  one does not know much about $\psi_{n}(1)$ except for the case $c=0$, for which $\psi_{n, 0}(1)=\sqrt{n+\frac 12}$. It is accepted, but not rigorously proved, that as a function of $c,$  $\psi_{n,c}(1)$ is maximum at $0$ (see \cite{Xiao}). We proved in \cite{Bonami-Karoui1} that, for $q\leq 2$, one has the inequality
\begin{equation}
\label{infty2}
\psi_{n}(1)\leq \kappa_1\chi_n^{1/4},\quad \kappa_1=\frac{5^{5/4}}{4}.
\end{equation}
We give here an approximated value of
$\psi_n(1)$ in terms of $\chi_n$  up to a relative error of order $O(1/(1-q)\sqrt{\chi_n})$.

The strategy of the proof is  simple. We start from the expression of $\psi_n$ given by  \eqref{uniform0} and set
\begin{equation}
\widetilde\psi_n(x)= \frac{\chi_{n}^{1/4}\sqrt{S(x)}J_0(\sqrt{\chi_n} S(x))}
{(1-x^2)^{1/4}(1-q x^2)^{1/4}}.
\end{equation}
We then  prove that the norm of $R_n$ is small and compute almost explicitly the $L^2$-norm of $\widetilde \psi_n$. The conclusion comes from these two computations.     

The next lemma gives bounds for the remainder $R_n$ in the $L^2-$norm. Recall that $\K$ denotes the
 complete Legendre elliptic integral of the first kind, given by \eqref{K}. \begin{lemma} \label{prop-L2}
Assume that   $n, c$ are such that
$q=c^2/\chi_n <1$. Then
\begin{equation}\label{Norm2Rnbis}
\|R_{n}\|_{L^2([0,1])}\leq \frac{\psi_n(1)}{\chi_n^{1/4}}\sqrt{\frac{2 \mathbf  K(\sqrt{q})}{\pi}}\alpha_q\, \varepsilon_n \Theta(\alpha_q\varepsilon_n).
\end{equation}
\end{lemma}
\begin{proof}
By Lemma \ref{errorU}, we have
\begin{equation}
\label{bound2Rn}
|R_{n}(x)|\leq \sqrt{\frac{2}{\pi}}\frac{\psi_n(1)}{\chi_n^{1/4}} \alpha_q \varepsilon_n\, \Theta(\alpha_q\varepsilon_n)\,  (1-x^2)^{-1/4}(1-q x^2)^{-1/4}.
\end{equation}
Moreover the $L^2(0,1)-$norm of the function $(1-x^2)^{-1/4}(1-q x^2)^{-1/4}$ is equal to $\sqrt{\mathbf  K(\sqrt{q})},$
which allows us to conclude for \eqref{Norm2Rnbis}.
\end{proof}
In order to evaluate the $L^2(0,1)-$norm of $\widetilde \psi_n$, we first define a constant related with Bessel functions.
 \begin{lemma} \label{kappa2}
The function $G(x)=\frac{x^2}{2}\left[(J_0(x))^2+(J_1(x))^2\right]-\frac{x}{\pi}$ is bounded on $[0, \infty)$. Moreover,
\begin{equation}
\label{Eeqq2.7}
\sup_{x>0}\left|\frac{x^2}{2}\left[(J_0(x))^2+(J_1(x))^2\right]-\frac{x}{\pi}\right|=\kappa_2=0.17203\cdots
\end{equation}
\end{lemma}
\begin{proof} The boundedness of $G$  comes from the fact that it has a finite limit at $\infty$, which is an easy consequence of the asymptotic expansion of Bessel functions. A careful study of the remainders associated with the previous asymptotic
approximations of Bessel functions proves that the  maximum of $G$ is attained in the interval $(0, 3)$.  Its monotonicity, as well as some numerical computations, are then used to get the  precise value of $\kappa_2.$ We leave the
details to the reader.
\end{proof}

We now prove the following lemma.
\begin{lemma} \label{norm-wide}
Under the above notations, for $0\leq q<1,$  we have
$$\|\widetilde\psi_n\|_{L^2([0,1])}^2 = \frac{\mathbf \mathbf  \mathbf  K(\sqrt{q})}{\pi}+\eta,\qquad |\eta|\leq
 {\kappa_2}\, \varepsilon_n.$$
 \end{lemma}
 \begin{proof}
Going back to the notations of Section 3. 1, that is, $U_1(s)=\sqrt{s} J_0(s)$,  one writes
$$\widetilde\psi_n (x) =\frac{U_1(\sqrt{\chi_n} S(x))}{(1-x^2)^{1/4}(1-q x^2)^{1/4}},\quad x\in [0,1].$$
If $x(s)$ denotes the inverse function of $S(x),$  one has
\begin{equation}
\label{norm2}
\|\widetilde\psi_n\|_{L^2([0,1])}^2 = \int_0^1 \frac{|U_1(\sqrt{\chi_n} S(x))|^2}{\sqrt{(1-x^2)(1-q x^2)}}\, dx=\sqrt{\chi_n}\int_0^{S(0)}\frac{s|J_0(\sqrt{\chi_n} s)|^2}{1- q x(s)^2}\, ds.
\end{equation}
Finally, after a last change of variables  one gets
\begin{equation}\label{Eeeqq0}
\|\widetilde\psi_n\|_{L^2([0,1])}^2
=\frac{1}{\sqrt{\chi_n}}\int_0^{\sqrt{\chi_n} S(0)}\theta(t) t (J_0(t))^2\, dt,
\end{equation}
where $$\theta(t)=\frac{1}{1- q x^2\left(\frac{t}{\sqrt{\chi_n}}\right)}, \qquad t\in [0, S(0)\sqrt{\chi_n}].$$
Since $0\leq q <1$ and $x(s)$ is decreasing and has values in $[0,1],$   the function $\theta(s)$ is smooth and decreasing on $[0, S(0)\sqrt{\chi_n}]$. It takes the value $1$ at $S(0)\sqrt{\chi_n}$. To estimate the previous quantity, we proceed as follows. We first note  (\cite{Andrews}) that
\begin{equation}\label{Eeqq2.4}
 \int_0^x t (J_0(t))^2\, dt = \frac{x^2}{2}\left[(J_0(x))^2+(J_1(x))^2\right],\quad x>0.
\end{equation}
So, if $G$ is defined as in Lemma \ref{kappa2}, we have that $xJ_0(x)^2=G'(x)+\frac 1\pi$. Integration by parts gives the equality
\begin{eqnarray*}
\sqrt{\chi_n}\|\widetilde\psi_n\|_{L^2([0,1])}^2&= &\frac{1}{\pi}\int_0^{\sqrt{\chi_n} S(0)}\theta (s)ds + \int_0^{\sqrt{\chi_n} S(0)} \theta(s)G'(s)  ds \\
&=& \frac{1}{\pi}\int_0^{\sqrt{\chi_n} S(0)}\theta (s)ds + G(\sqrt{\chi_n} S(0))-\int_0^{\sqrt{\chi_n} S(0)} \theta'(s)G(s)  ds .
\end{eqnarray*}
We use \eqref{Eeqq2.7} to bound  the second and the third term. The second one is directly bounded by $\kappa_2.$ Since $\theta'$ is non positive,  the last term is bounded by $\kappa_2 (\theta (0)-\theta (1))=\frac{\kappa_2}{1-q}-\kappa_2$.  Now by using the substitution $s= \sqrt{\chi_n} S(x),$ one gets
$$\frac 1  {\sqrt{\chi_n}}\int_0^{\sqrt{\chi_n} S(0)} \theta (s)ds=\int_0^{1} \frac {d x}{\sqrt{1-qx^2}\sqrt{1-x^2}}=\mathbf K(\sqrt q).$$
By collecting everything together, one concludes that
\begin{equation*}
\left|\| \widetilde\psi_n\|_{L^2([0,1])}^2-\frac{\mathbf  K(\sqrt{q})}{\pi} \right|\leq {\kappa_2}\, \varepsilon_n.
\end{equation*}
 \end{proof}
 As a corollary, we have the following bounds for the norm of $\widetilde\psi_n$. If $\beta_q\varepsilon_n<1,$ where
 \begin{equation}\label{beta}
  \beta_q= \frac{\pi\kappa_2}{\mathbf  K(\sqrt{q})}\leq 2\kappa_2,
  \end{equation}
then we have
 \begin{equation}\label{boundstpsi}
\sqrt{\frac{\mathbf  K(\sqrt{q})}{\pi}} (1-\beta_q\varepsilon_n)^{1/2}\leq \| \widetilde\psi_n\|_{L^2([0,1])}\leq\sqrt{\frac{\mathbf  K(\sqrt{q})}{\pi}} (1+\beta_q\varepsilon_n)^{1/2}.
 \end{equation}
 At this point, we see that the  left hand side estimate is  interesting only if $\beta_q\varepsilon_n$ is sufficiently small.  We give now a slightly stronger but also a flexible  assumption that will be sufficient for the inequalities to come,
 \begin{equation}
 \label{Condition1}
 (1-q)\sqrt{\chi_n}  \geq 4.
  \end{equation}
Note that by using Proposition \ref{comparison}, we may write the above condition
in terms of $n$ and $c$ directly, without involving $\chi_n$.
Under Condition \eqref{Condition1}, we have
\begin{equation}\label{numerics}
\alpha_q \varepsilon_n<0. 375,\qquad \qquad \beta_q \varepsilon_n< 0.086, \qquad \qquad (1-\beta_q \varepsilon_n)^{1/2}>0. 96.
\end{equation}
 The following theorem provides us with an approximation of $\psi_n(1)$.
\begin{theorem}
Let   $n, c$ be such that $(1-q)\sqrt{\chi_n}   \geq 4$. Then
\begin{equation}\label{boundsA}
\chi_n^{1/4} \sqrt{\frac{\pi}{2 \mathbf  K(\sqrt{q})}}\left(1-\eta \, \varepsilon_n\right) \leq \psi_n(1) \leq \chi_n^{1/4} \sqrt{\frac{\pi}{2 \mathbf  K(\sqrt{q})}}
 \left(1+\eta'\, \varepsilon_n\right),
\end{equation}
where we may take $\eta=2.75$, $\eta'=10.78$.
\end{theorem}

\begin{proof}
Let $A=\psi_{n, c}(1)\chi_n(c)^{-1/4}$ as before.
By the triangular inequality,  we have
\begin{equation}\label{IneqA1}
\left|\frac{1}{\sqrt{2}}-A\|\widetilde\psi_n\|_{L^2([0,1])}\right|=\left|\|\psi_n\|_{L^2([0,1])}-A\|\widetilde\psi_n\|_{L^2([0,1])}\right|
\leq \| R_{n}\|_2.
\end{equation}
By using the previous equality and  \eqref{Norm2Rnbis}, one gets
\begin{equation}\label{IneqA2}
 A\left(\|\widetilde \psi_n\|_{L^2(0, 1)}-\sqrt{\frac{2 \mathbf  K(\sqrt{q})}{\pi}}\alpha_q\, \varepsilon_n\, \Theta(\alpha_q\varepsilon_n)\right) \leq \frac 1{\sqrt{2}} \leq
A\left(\|\widetilde \psi_n\|_{L^2(0, 1)}+\sqrt{\frac{2 \mathbf  K(\sqrt{q})}{\pi}} \alpha_q\, \varepsilon_n \, \Theta(\alpha_q\varepsilon_n)\right).
\end{equation}
Moreover, from (\ref{boundstpsi}), we have
\begin{equation}\label{IneqA3}
\sqrt{\frac{2 \mathbf  K(\sqrt{q})}{\pi}}\left((1-\beta_q\varepsilon_n)^{1/2}-\sqrt 2 \alpha_q\, \varepsilon_n\, \Theta(\alpha_q\varepsilon_n)\right) \leq A^{-1} \leq \sqrt{\frac{2 \mathbf  K(\sqrt{q})}{\pi}}\left((1+\beta_q\varepsilon_n)^{1/2}+\sqrt 2\alpha_q\, \varepsilon_n\, \Theta(\alpha_q\varepsilon_n)\right).
\end{equation}
It follows immediately that we have \eqref{boundsA} as soon as
$$\eta \varepsilon_n \geq 1 -\frac{1}{\sqrt{1+\beta_q\varepsilon_n}+\sqrt{2}\alpha_q\varepsilon_n\, \Theta(\alpha_q\varepsilon_n)}.$$ This is in particular the case when
\begin{equation}\label{eta}
 \eta >\frac {\beta_q}2+\sqrt 2\alpha_q \, \Theta(\alpha_q\varepsilon_n).
 \end{equation}
 Taking into account the estimates given in \eqref{numerics} and the fact that $\alpha_q<1.5$, we get  the stronger sufficient condition $\eta > 2.75$.
The same method for $\eta'$ gives
\begin{equation} \label{eta2}
\eta'> \frac {\beta_q+\sqrt{2}\alpha_q\, \Theta(\alpha_q\varepsilon_n) }{1-\beta_q\varepsilon_n-\sqrt{2}\alpha_q\varepsilon_n
\, \Theta(\alpha_q\varepsilon_n)}.
\end{equation}
 Again, by taking into account the estimates given in \eqref{numerics}, we get the stronger sufficient condition    $\eta' > 10.78$.
\end{proof}

\begin{remark}
Remark that for a  fixed $\varepsilon_n,$ the two inequalities \eqref{eta} and \eqref{eta2} give better estimates than the numerical values given in the statement of the theorem. Moreover numerical tests show that the relative error in estimating $A$ by ${\displaystyle \sqrt{\frac{\pi}{2 \mathbf  K(\sqrt{q})}}}$ is much smaller
than this theoretical error.
\end{remark}

\medskip

This last theorem allows us to improve the estimate given in \eqref{infty2}, at least asymptotically.
By using \eqref{chi+}, we find that
\begin{equation}
\label{bound2psi2}
\psi_n(1)\leq (n+1)^{1/2}\sqrt{\frac{\pi^2}{4 \mathbf  E(\sqrt{q})\mathbf  K(\sqrt{q})}}\left(1+\eta' \, \varepsilon_n\right).
\end{equation}
\begin{remark} This inequality does not imply the one that has been conjectured in \cite{Xiao} from numerical evidence, namely $\psi_n(1)\leq\sqrt{n+\frac 12}$.
But there are many values of $n, c$ for which it is better: by Cauchy-Schwarz Inequality  we know that $\frac{\pi^2}{4 \mathbf  E(\sqrt{q})\mathbf  K(\sqrt{q})}<1$ for $q>0$. Moreover this quantity goes to $0$ as $q$ goes to $1$. \end{remark}
Next, replacing $\psi_n(1)$ by its approximation in Proposition \ref{uniform-prop}, we get the following corollary.
\begin{corollary}\label{main}
There exist  constants  $C_1$ and $C_2$ such that for all $n, c$ for which
$(1-q)\sqrt {\chi_n (c)} \geq 4,$ we have, for $0\leq x\leq 1$
\begin{equation}\label{uniform}
    \psi_{n}(x)=  \sqrt{\frac{\pi}{2 \mathbf  K(\sqrt{q})}}\frac{\chi_n^{1/4}\sqrt{S(x)}J_0(\sqrt{\chi_n} S(x))}
{(1-x^2)^{1/4}(1-q x^2)^{1/4}}+\widetilde R_{n}(x)
\end{equation}
with
\begin{equation}\label{bounds1-tilde}
|\widetilde R_{n}(x)|\leq
 C_1 \varepsilon_n\sqrt{\frac{1}{K(\sqrt{q})}}\min\left(\chi_n^{1/4}, \, (1-x^2)^{-1/4}(1-qx^2)^{-1/4}\right).
\end{equation}
Moreover, we have
\begin{equation}\label{boundsA-tilde}
\|\widetilde R_{n} \|_{L^2([0,1])}\leq  C_2 \, \varepsilon_n.
\end{equation}
\end{corollary}

\begin{proof}  We write $\widetilde R_{n} = R_{n}+\left(A-\sqrt{\frac{\pi}{2 \mathbf \mathbf  K(\sqrt{q})}}\right)\widetilde \psi_n$.
By using (\ref{bounds1}) and (\ref{Condition1})  one checks that $|R_{n}(x)|$ satisfies
a similar bound as the one given by (\ref{bounds1-tilde}). Moreover, since  $|J_0(s)|\leq 1$, we know by using \eqref{equS} that $\widetilde \psi_n(x) $ is bounded by $\sqrt 2 \chi_n^{1/4}$. Using also the bound of $\sqrt{s} | J_0(s)|$ on  $\mathbb R_+,$ one gets
$$|\widetilde \psi_n(x)| \leq \min\left(\sqrt{2}\chi_n^{1/4},\, \sqrt{2/\pi} (1-x^2)^{-1/4}(1-qx^2)^{-1/4}\right).$$
The proof of (\ref{bounds1-tilde}) follows from the previous inequality and from \eqref{boundsA}. This later implies that
the second term satisfies also a similar  bound to the one
given in (\ref{bounds1-tilde}). The estimates in $L^2$ follow from \eqref{Norm2Rnbis}, Lemma \ref{norm-wide} and (\ref{bounds1-tilde}) again.
\end{proof}

\subsection{An improved error bound for the eigenvalues $\chi_n.$}

In this subsection, we use the WKB approximation of the function $U$ given by Lemma \ref{errorU} to improve the bounds given for the eigenvalue $\chi_n$.  This  result is contained in the following theorem. We first recall that under the notations of Section 2, we have
\begin{equation}\label{chitilde}
\sqrt{\widetilde \chi_n(c)}= c \tilde \Phi\left(\frac{\pi(n+\frac 12)}{2c}\right).
\end{equation}
\begin{theorem}
There exist  two  constants $\kappa, \kappa'$ such that, when $(1-q)\sqrt{\chi_n}\geq \kappa',$ we have
\begin{equation}\label{approxChi}
|\sqrt{\widetilde \chi_n(c)}-\sqrt{ \chi_n(c)}|\leq  \frac{\kappa}{(1-q)\sqrt{\chi_n }}.
\end{equation}
\end{theorem}
Remark that this justifies the choice of the middle point $n+\frac 12$ in the approximation of $\sqrt{\chi_n}$ by formula \eqref{chitilde}.
\begin{proof}
The proof is based on the comparison of the  (normalized) function  $V=\psi_n(1)^{-1}\chi_n^{1/4}U$ with the function
$$W(s)=(\sqrt{\chi_n} s)^{1/2} J_0(\sqrt{\chi_n} s).$$
When $(1-q)\sqrt{\chi_n}\geq 1$, we know by Lemma \ref{errorU} that $V$ and $W$ differ by $\frac{\gamma }{(1-q)\sqrt{\chi_n}}$ on the interval $[0, S(0)]$ (the lemma gives an explicit constant $\gamma$, but from this point we do not try to track constants). Moreover both functions have alternative positive and negative local extrema, with increasing absolute values (see for instance  \cite{Szego} p. 166 and also \cite{Bonami-Karoui1}; for the function $V$ it is a consequence of Proposition \ref{increasing}). It is easy to characterize the point $S(0)$ for the function $V$. Indeed, when $n$ is odd, the function $\psi_n$ is also odd and  $S(0)$ is the $\frac{n+1}2$-th zero of $V$. When $n$ is even, then $S(0)$ is the  $\frac n2$-th local extremum of $V.$ We have used the fact that zeros of $V$ correspond to zeros of $\psi_n$ through the change of variable given by $S$ (except for $s=0$, which we do not count), so that $V$ has $\frac{n+1}2$ zeros in $(0, S(0)]$ when $n$ is odd and  $\frac n2$ zeros when $n$ is even. We claim that the two functions $V$ and $W$ have nearby  zeros and extrema. This allows us to get  an approximation of $\sqrt {\chi_n} S(0)$.
\medskip

The proof is easier for $n=2m+1$ and we  first do it in this case. Let $m_0\approx 0.767$ be the value of the first local maximum of $W$ and assume that  $(1-q)\sqrt{\chi_n}\geq 5\gamma$. The maxima of $|W|$ are all larger than $m_0$ and smaller than $m_\infty=\sqrt{2/\pi}$, while, because of the fact that their difference is at most $1/4$, the maxima of $|V|$ lie in the interval $(1/2, 1)$. It is classical to deduce from this study that the $k$-th zero of one of the functions lies between two consecutive extrema of the other one (see for instance \cite{Widom}). We use this to conclude that $W$ has exactly $m$ local extrema on $(0, S(0)]$ and that $\sqrt{ \chi_n} S(0)$ belongs to the interval $(s_m, s_{m+1})$. Here $s_k$ is the value of the variable $s$ for which $W$ takes its $k$-th local extremum.

At this moment we use also the asymptotic behavior of the Bessel function $J_0$ to write that, for $s\geq S(0)/2,$
  $$|(\sqrt{\chi_n}s)^{1/2}J_0(\sqrt{\chi_n} s)-\sqrt{\frac 2\pi}\cos(\sqrt{\chi_n} s-\frac\pi 4)|\leq \frac {\gamma'}{\sqrt{\chi_n}}.$$
  If we set $T(s)=\sqrt{\frac 2\pi}\cos(\sqrt{\chi_n} s-\frac\pi 4)$, then $T$ and $W$, as well as $T$  and $V$, differ from a quantity that is much smaller than their local extrema under the condition that $(1-q)\sqrt{\chi_n}\geq \kappa'$, with $\kappa'$ large enough. We now assume that this is satisfied. So both $ V$ and
 $W$ have their $m+1$-th zero between the two same consecutive extrema of $T$. We know that the $m+1$-th zero of $J_0$ belongs to the interval $(m\pi +\frac \pi 4, (m+1)\pi +\frac \pi 4)$. So $\sqrt{ \chi_n} S(0)$ belongs to the same interval.

 \medskip

If $\delta=\sqrt{ \chi_n} S(0)-m-\frac{3\pi}4$, we have $\delta\in(-\pi/2, +\pi/2)$. Moreover, $|\sin (\delta)|=\sqrt{\frac \pi 2}|T(\sqrt{ \chi_n} S(0))|\leq \sqrt{\frac \pi 2}\frac{\gamma+\gamma'}{(1-q)\sqrt{ \chi_n}}$, using the fact that $V$ vanishes at this point. It is elementary that this implies the inequality $|\delta|\leq \frac{\pi\kappa}{2(1-q)\sqrt{ \chi_n}}$ for $\kappa=(\frac{\pi}{2})^{1/2}(\gamma+\gamma').$

 We have proved  that
 $$ |\sqrt{\chi_n }\E(\sqrt q)-\frac \pi 2(n+\frac 12)|\leq \frac{\pi\kappa}{2(1-q) \sqrt{\chi_n}}.$$ By dividing by $c$, composing with the function $\widetilde\Phi$, and using the fact that its derivative is bounded by $\frac{2}{\pi}$ and finally, by multiplying by $c,$ one concludes that if $(1-q)\sqrt{\chi_n}\geq \kappa',$ then we
have
  $$|\sqrt{\chi_n}-\sqrt{\widetilde \chi_n}| \leq \frac{\kappa}{(1-q) \sqrt{\chi_n}}.$$

  \smallskip
    It remains to adapt the proof to even values of $n$. Now $U$ does not vanish at $S(0)$ and the same role will be played by the largest zero $s_0$ of $U$. It is elementary to see that the proof is the same as in the previous case, once we have proved that $s_0=S(0)-\frac {\pi}{ 2\sqrt {\chi_n}}+O(1/\chi_n)$. So let us prove this last fact.
 It is easier to change the variable as in Section 2 and consider the first zero $S(0)-s_0>0$  of the even solutions of Equation \eqref{tilde}. By parity, $s_0-S(0)$ is the next zero on the left.
  We use Sturm comparison theorem between equation \eqref{tilde} and the equation $Y''+\left(\chi_n+\frac {q+1}2\right) Y=0$ to obtain that
 $S(0)-s_0\geq \frac {\pi}{ 2\sqrt {\chi_n+\frac{q+1}2}}$ as in Section 2.
  To have an upper bound of the quantity $S(0)-s_0,$  we also use Sturm comparison theorem with the equation $Y''+\left(\chi_n+B\right) Y=0,$
 where $B$ is an appropriate bound of  $\widetilde \theta,$  or equivalently, where the function $\theta\circ S$ is bounded by $B$  between $0$ and the first zero of $\psi_n$. Osipov has proved in \cite{Osipov} that the first positive zero of $\psi_n$ lies before $\frac {\pi}{ 2\sqrt {\chi_n +1}}$. At this point it is sufficient to consider the expression of $\theta\circ S$ given in \eqref{h1S}. It is bounded by $\frac 54 (1-\frac{\pi^2}4\chi_n)^{-2}$, from which we conclude.
  \end{proof}

 \begin{remark}
 From \eqref{approxChi} the approximation error caused by replacing $\sqrt{\chi_n}$ by  $\sqrt{\widetilde \chi_n}$, is of the same order as the one obtained from  replacing $\psi_n$ by its WKB approximation, up to the factor $\K(\sqrt q)^{-1/2}$. This probably explains the accuracy of numerical tests  in which PSWFs are replaced by the main term of Corollary \ref{main}, this last one being computed with
$\sqrt{\widetilde \chi_n}$ instead of  $\sqrt{\widetilde \chi_n}$.  \end{remark}

\section{Uniform estimates  for the PSWFs  at  end points for $q$.}

Our previous results do not extend to the value $q=1$, mainly because of the fact that the function $F$, which has been introduced in Section 3
is no longer  continuous on the whole interval $[0, S(0)]$. At this moment we do not know how to deal with all values of $q<1$ at the same time, but we concentrate here on the case $q=1$, where we can nevertheless develop uniform approximation.  The underlying idea is that the previous methods are valid on the interval $(0, 1-\sqrt{\chi_n}^{-1})$, while a priori estimates allow to know the behavior of $\psi_{n}$ on the missing interval. This is not only valid for $q=1$ but for values of $q$ that are very close to $1$, for which the change of variable related to $1$ is still relevant.

In a second subsection, we give a WKB approximation of $\psi_n$ in terms of the normalized Legendre polynomial $\overline P_n$. This may be helpful for $c$ small and $n$ not too large. Otherwise Legendre polynomials themselves are very well approximated in terms of $J_0$.

\subsection{Uniform estimates  for the PSWFs when $c^2 \approx\chi_n.$}
We  see in this paragraph that the method that we have used for $q<1$ holds also  for $q=1$ and, up to some extent, when $\chi_n$ and $c^2$ are very close. Let us first recall that (see for example  \cite{Wang})
$$\partial_c(c^2-\chi_n(c))=2c \int_{-1}^{+1}(1-t^2)|\psi(t)|^2dt$$
so that, for a fixed value of the positive integer $n,$ there is exactly one value of $c$ for which $q=1$.
We go back to the equation \eqref{eqq2.1} and use the change of function $U(1-x)=\sqrt{1-x^2}\psi_n(x)$ that coincides with the previous change when $q=1$. The equation for $U$ is now
\begin{equation}\label{eqdif2}
U''(s)+\left(\frac{1}{4s^2(1-s/2)^2}+\chi_n +\frac{\chi_n-c^2}{2s(1-s/2)}\right) U(s)=0, \quad 0\leq s\leq 1.
\end{equation}
As before, we write this equation as
$$U''(s)+\left(\frac{1}{4s^2}+\chi_n \right)U(s) = F(s) U(s).$$
A straightforward computation leads to  $|F(s)|\leq\frac 12 + |\gamma_n| +\frac{|\gamma_n|}{s}$, where
\begin{equation}\label{gamma}
\gamma_n=\frac{2(1-q)\chi_n +1}4.
\end{equation}
We let $\delta_n= \frac 12 + 2|\gamma_n|$ and only use the inequality $|F(s)|\leq \frac{|\delta_n|}{s}$ to simplify expressions, even if this leads to weaker estimates. Remark in particular that, when $\gamma_n$ vanishes, the function $F$ is bounded and the results of the last section are  directly adapted.

We do not restrict ourselves  to this case but go on with the same kind of proof. For simplification, we use the same notations. We write
$$\psi_{n}(x)=  \frac{A \chi_n^{1/4}\,J_0(\sqrt{\chi_n}(1-x))}{\sqrt{1+x}} + R_n(x),$$
with $A=\sqrt{2}\psi_n(1)\chi_n^{-1/4}$. We have
$$\sqrt{1-x^2}R_n(x)= W(1-x)= \frac{1}{\sqrt{\chi_n}} \int_0^{1-x} K_n(1-x,t) F(t) U(t)\, dt.$$
As in the last section the kernel $|K_n|$ is bounded by $1$. We claim that $W$ satisfies the inequality
\begin{equation}\label{Wmaj}
|W(s)|\leq \delta_n \chi_n^{-1/2} \left\{\begin{array}{ll}
 2A(s\sqrt{\chi_n})^{1/2}  &\mbox{ if\ \  } s\leq \chi_n^{-1/2}\\ 2A+ (\log \sqrt{\chi_n})\sup_{s\in[0,1]}|U(s)| &\mbox{ otherwise } .\end{array}\right.
\end{equation} Indeed, we have the inequality
$$|W(s)|\leq \delta_n \chi_n^{-1/2}\int_{0}^{s} \frac{|U(t)|}{t} dt= \delta_n \chi_n^{-1/2}\int_{1-s}^{1} \sqrt{\frac {1+t}{1-t}} |\psi_n(t)| dt.$$
 Recall that the maximum of $|\psi_n|$ is attained at $1$. We use the last expression and the fact that $|\psi_n(t)|\leq \frac A{\sqrt 2}(\chi_n)^{1/4}$ to conclude for the first bound. For the second one, we cut the integral into two parts. From $0$ to $\chi_n^{-1/2}$,  we use the previous inequality. From $\chi_n^{-1/2}$ to $s$,  we just conclude directly
by using  the first expression.

\medskip

We will not use Olver's estimates but proceed as in  the proof of \eqref{eqU} given at the end of Section 3. 1. As a consequence of \eqref{Wmaj}, we have the inequality
 $$\sup_{s\in[0,1]}|U(s)|\leq \sqrt{\frac{2}{\pi}}A+ \delta_n \chi_n^{-1/2}(2A+(\log \sqrt{\chi_n})\sup_{s\in[0,1]}|U(s)|.$$
From now on we assume  that $\delta_n \chi_n^{-1/2}\log \sqrt{\chi_n}<1/2$, so that we conclude from  the last inequality that
$\sup |U(t)|\leq 4A$.
In the sequel, we do not give explicit bounds for uniform  constants $\kappa.$  We have the following inequality,
  which plays the role of the estimate given by \eqref{bounds1}.
 \begin{equation}
| R_n(x)|\leq \kappa A \delta_n \chi_n^{-1/2} \log \sqrt{\chi_n} \min ({\chi_n}^{1/4}, (1-x^2)^{-1/2}).
 \end{equation}
 Moreover, it follows from the expression of $R_n$ that
 $$\|R_n\|_{L^2([0, 1])}\leq \kappa  A \delta_n \chi_n^{-1/2}
 \left(\log (\sqrt{\chi_n})\right)^{3/2}. $$
 From this point, the same method as in the last section can be used. We have  to find an equivalent of the $L^2(0,1)-$norm of the main term.
 This is based on the following lemma.

 \begin{lemma}
 Let  $\gamma$ be the Euler constant. Then, for any real $x>0,$ we have
  \begin{equation}
  \label{integralJ0}
  \int_0^{x}J_0(t)^2 dt = \frac{1}{\pi}\left(\log(x)+\gamma+3\log 2\right)+ \epsilon_n,\quad
  |\epsilon_n|\leq \frac{0.4}{x}.
  \end{equation}
  \end{lemma}

 \begin{proof}
 The proof  follows from
  Lemma \ref{kappa2} and from the  following identity, given in \cite{Glasser}
  $$\int_0^{\infty} \left(J_{0}(t)^2-\frac{1}{1+t}\right)\, dt = \frac{1}{\pi}\left(\gamma+3\log 2\right).$$

  We can write
 \begin{eqnarray*} \int_0^x J_0^2(t)\, dt &=&\frac{1}{\pi}\left(\gamma+3\log 2+ \log (1+x)\right) -\int_x^{\infty} \left(J_0^2(t)-\frac{1}{\pi (t+1) }\right)\, dt\\
 &=&\frac{1}{\pi}\left(\gamma+3\log 2+ \log (x)\right)-\int_x^{\infty} \left(J_0^2(t)-\frac{1}{\pi t }\right)\, dt. \end{eqnarray*}
 If as before, ${\displaystyle G(x)= \frac{x^2}{2}\left[(J_0(x))^2+(J_1(x))^2\right]-\frac{x}{\pi}},$ then   a simple integration by parts gives us
 $$\int_x^{\infty}\left(J_0^2(t) -\frac{1}{\pi t}\right)\, dt =\int_x^{\infty}\frac{1}{t}\left( t J_0^2(t) -\frac{1}{\pi }\right)\, dt= -\frac{G(x)}{x} +\int_x^{\infty} \frac{G(t)}{t^2}\, dt.$$
 Hence, we have
 $$\left|\int_x^{\infty} \left(J_0^2(t)-\frac{1}{\pi t }\right)\, dt\right|\leq \frac{2 \kappa_2}{x},\quad x>0.$$
 \end{proof}

\begin{lemma} Let $\widetilde \psi(x)= \frac{\chi_n^{1/4}\,J_0(\sqrt{\chi_n}(1-x))}{\sqrt{1+x}}$. Then, for   $\beta=4 \log 2+\gamma$, where $\gamma$ is the Euler constant, we have
\begin{equation}
 \label{norm2tpsi}
 \|\widetilde \psi\|_2^2=\frac{\log(\sqrt{\chi_n}) +\beta}\pi  +\epsilon_n,\quad |\epsilon_n|\leq
 \frac{0.8}{\sqrt{\chi_n}}.
 \end{equation}
 \end{lemma}

 \begin{proof}
 After a change of variable, we have
 $$
 \|\widetilde \psi\|_2^2=\int_0^{\sqrt{\chi_n}}\frac{|J_0(t)|^2}{1- t/(2\sqrt{\chi_n})}\, dt=\int_0^{\sqrt{\chi_n}}|J_0(t)|^2 dt+\frac 1{2\sqrt{\chi_n}}\int_0^{\sqrt{\chi_n}}\frac{t|J_0(t)|^2}{1- t/(2\sqrt{\chi_n})}\, dt.
 $$
 From  \eqref{integralJ0}, the first term in place is given by
 $$
 \int_0^{\sqrt{\chi_n}}|J_0(t)|^2 dt = \frac{1}{\pi}\left(\log(\sqrt{\chi_n})+\gamma+3\log 2\right)+ \epsilon_n,\quad
 |\epsilon_n|\leq \frac{0.4}{\sqrt{\chi_n}}.
 $$
To estimate the second term in (\ref{norm2tpsi}), one can proceed exactly as in the proof of Lemma \ref{norm-wide} with
now ${\displaystyle \theta(t)= \frac{1}{1-t/(2\sqrt{\chi_n})}}$ and find the quantity $\frac{\log 2}\pi +\epsilon_n$ with $|\epsilon_n|
\leq 2\kappa_2/\sqrt{\chi_n}. $
 \end{proof}
 This allows us to state the following result, which may be seen as an extension of the   previous section  in the context $\chi_n\approx c^2.$

 \begin{theorem}
Let $\varepsilon>0$ be given. There exists  a constant $\kappa$  such that, for $\delta_n \chi_n^{-1/2}\log\sqrt{\chi_n}\leq \kappa,$ one can write
\begin{equation}\label{eq2.1}
\psi_{n}(x)= \left(\frac{\pi}{\log(\sqrt{\chi_n})+\beta}\right)^{\frac 12} \frac{\chi_n^{1/4} J_0(\sqrt{\chi_n} (1-x))}{\sqrt{1+x}} + \widetilde R_n(x),
\end{equation}
with $$|\widetilde R_n(x)|\leq  \varepsilon \left(\log (\sqrt{\chi_n})\right)^{-1/2}\min\left(\chi_n^{1/4}, \frac{1}{\sqrt{1-x^2}}\right), \qquad \|\widetilde R_n\|_2\leq \varepsilon.$$
\end{theorem}

We  do not give more details. Remark that we also proved that under the conditions given on $n, c$, we have the asymptotic value
$\left(\frac\pi{2\log \sqrt{\chi_n}+\beta} \right)^{\frac 12} \chi_n^{1/4}$ for $\psi_n(1)$. When ${ \gamma_n=\frac{2(1-q)\chi_n+1}{4}=0,}$ it is sufficient to have a condition that does not involve a logarithm, namely $\delta_n \chi_n^{-1/2}\leq \kappa.$

\subsection{Uniform estimates  for the PSWFs  when $c^2/\chi_n \approx 0.$}

Let $\overline{P}_n=\psi_{n, 0}$ be the normalized Legendre polynomials, so that $\|\overline{P}_n\|_{L^2[-1,1]}=1.$ Section 3 can be used to obtain uniform estimates of $\overline{P}_n$. This kind of estimates for Legendre polynomials have been known for a long time, see \cite{Szego}. In this case, we have $S(x)=\arccos (x)$ and it is simpler to use $n+1/2$ instead of $\chi_n$ in \eqref{eqq2.5}, so that
uniform estimates may be written as
\begin{equation}\label{legendre-bis}
\overline P_n(\cos \theta)= (n+1/2)^{1/2} \left(\frac{\theta}{\sin \theta}\right)^{1/2}J_0\left((n+1/2) \theta\right)
+ O\left(\frac 1n\right)
    \end{equation}
for $0\leq \theta\leq \frac \pi 2$.
Precise estimates of the remainder, which improve ours, are given in \cite{Gatteschi}. In this paragraph, we use  the same method as in Section 3  to approximate $\psi_n$ by $\overline{P}_n$ when $c$ is close to $0$. The main result is given by the next proposition. The first statement expresses the fact that $\psi_{n}(1)$ is close to its value for $c=0$, that is, $\sqrt{n+1/2}$.
\begin{proposition}  For all $n\in \mathbb N$ and $c\geq 0$ we have the inequalities
\begin{equation}\label{one}
\left|\frac{\psi_{n, c}(1)}{\sqrt{n+1/2}}-1\right|\leq \frac{c^2}{\sqrt {3} (n +1/2)},
\end{equation}
\begin{equation}\label{secondsup}
\sup_{x\in [-1,1]}
\left|\psi_{n,c}(x)-\overline{P}_n(x)\right|\leq
\frac{ c^2}{\sqrt {3 (n +1/2)}}\left(1+\frac{\sqrt { 3/2}}{\sqrt{n +1/2}}\right).
\end{equation}
\end{proposition}

\begin{proof}
 We now rewrite the equation  as follows,
$$(1-x^2) \psi_n''-2 x\psi_n'+n(n+1)\psi_{n}=\left(
n(n+1)-\chi_n+c^2 x^2\right)\psi_n.$$ The homogeneous  equation
$$(1-x^2) \psi_n''-2 x\psi_n'+n(n+1)\psi_{n}=0$$
has the two linearly independent solutions  $\overline P_n$ and $Q_n$, where  $Q_n$ is the Legendre function of the second kind. Recall that (see \cite{Abramowitz}), because of the normalization for $\overline{P_n}$, the Wronskian
 is given by $W(\overline{P}_n, {Q_n})(x)=\frac{n+1/2}{1-x^2}.$
The absolute value of the function $G=n(n+1)-\chi_n+c^2 x^2$ is bounded  by $c^2$. By the method of variation of constants, we can write, for $x>0$,
 $$\psi_n(x)=A \overline{P}_n(x)+\frac 1{n+1/2}\int_x^1 L_n(x, y) \sqrt{1-y^2}  G(y)\psi_n(y)\,
dy,$$
where we have used the notation
\begin{equation}
\label{Ln}
L_n(x,y)=\sqrt{1-y^2}\left(\overline{P}_n(x)
Q_n(y)-\overline{P}_n(y)
Q_n(x)\right).
\end{equation}
As in the other cases, the behavior at $1$ has been used to see that there is no term in $Q_n$.

We have the following lemma, which is the equivalent of \eqref{Ineq10} for Bessel functions in the present context.
\begin{lemma}\label{legendre}
We have the inequality, valid for all $n>0$ and $x<1$,
\begin{equation}\label{legendre-ineq}
\sqrt{1-x^2}\left(\overline{P}_n(0)^{-2} \overline{P}_n(x)^2+  \overline{P}_n(0)^2 Q_n(x)^2\right)\leq 1.
\end{equation}
\end{lemma}
Let us take this lemma for granted and go on for the proof. By Cauchy-Schwarz inequality, this implies in particular that, for $0\leq x\leq y\leq 1$ the kernel $|L_n|$ is bounded by $1$. The consideration of the behavior of each term when $x$ tends to $1$ implies that $B$ is equal to $0$. Moreover, by using the fact that $\psi_n$ has norm $1/\sqrt{2}$ in $L^2([0,1])$ and  Cauchy-Schwarz inequality, one gets  the inequality
$$|\psi_n(x)-A\overline P_n(x)|\leq   \frac {c^2}{\sqrt 2 (n+1/2)}(1-x). $$
Since the function  $\overline P_n$ has also $L^2([0,1])-$norm $1/\sqrt{2},$ then we have
$|1-A|\leq   \sqrt{\frac{1}{3}}\frac{c^2}{n+1/2}. $ This gives \eqref{one}.
In view of \eqref{secondsup},  we have
\begin{eqnarray*}
\sup_{x\in [0,1]} |\psi_n(x)-\overline{P_n}(x)|&\leq&
\sup_{x\in [0,1]}
|\psi_{n}(x)-A\overline{P}_n(x)|+|A-1|
\sup_{x\in [0,1]} |\overline{P}_n(x)|.
\end{eqnarray*}
We conclude by using the previous inequalities and the fact that $|\overline P_n|$ is bounded by its value at $1$.
\end{proof}

It remains to prove Lemma \ref{legendre}.
\begin{proof}[Proof of Lemma \ref{legendre}]
We have seen that
$\sqrt {1-x^2}|\overline {P}_n(x)|^2\leq |\overline{P}_n(0)|^2+ \left(\frac{|\overline{P}_n'(0)|}{n+1/2}\right)^2.$
Depending on the parity, only the first term or the second term is non-zero in the right hand side. We shall assume that $n$ is even, but the proof would be similar for odd values. This inequality is valid for all solutions of the homogeneous equation.
In particular, it is valid for the particular solution $c \overline {P}_n(x)+ d Q_n(x),$ with $c, d \in \mathbb R.$
 This means that  for $n$ even, $Q_n(0)= \overline{P}'_n(0)= 0$ and
\begin{equation}\label{cylindral}
\sqrt{1-x^2}({c \overline{P}_n(x)+ d Q_n}(x))^2\leq c^2P_n(0)^2+  \left(\frac{{d^2Q_n'}(0)}{n+1/2}\right)^2.
\end{equation}
Using the fact that $$\frac {Q_n'(0)\overline{P}_n(0)}{n+1/2}= \frac {W(\overline{P}_n, {Q_n})(0)}{n+1/2}=1,$$
the previous inequality becomes
\begin{equation}\label{cylindral2}
\sqrt{1-x^2}({c \overline{P}_n(x)+ d Q_n}(x))^2\leq c^2 \overline{P}_n(0)^2+  d^2 \frac{1}{\overline{P}_n(0)^2}.
\end{equation}
This inequality is valid for all $c, d$. We conclude by taking the particular choices $c= \overline{P}_n(0)^{-2}\overline{P}_n(x), d=\overline{P}_n(0)^2Q_n(x)$ and dividing both sides of the inequality that we obtain by the right hand side.
\end{proof}


\begin{thebibliography}{9999}

\bibitem{Abramowitz} M. Abramowitz and I. A. Stegun, Handbook of mathematical
 functions, Dover Publication, INC, New York, 1972.

\bibitem{Andrews} G. E. Andrews, R. Askey and  R. Roy, Special Functions,
Cambridge University Press, 2000.

\bibitem{Bonami-Karoui1} A. Bonami and A. Karoui,  Uniform bounds of prolate spheroidal wave functions and eigenvalues decay,
{\it C. R. Math. Acad. Sci. Paris,}  {\bf 352} (2014), 229--234.

\bibitem{Bonami-Karoui2} A. Bonami and A. Karoui, Spectral decay of  Sinc kernel operators and approximation by
Prolate Spheroidal Wave Functions, available at http://arxiv.org/abs/1012.3881, submitted for publication (2014).

\bibitem{Boyd1} J. P. Boyd, Prolate spheroidal wave functions as an alternative to Chebyshev and Legendre
 polynomials for spectral element and pseudo-spectral algorithms, {\it J. Comput. Phys.} {\bf 199}, (2004), 688--716.

\bibitem{Boyd2} J. P. Boyd, Approximation of an analytic function on a finite real interval by a bandlimited function
and conjectures on properties of prolate spheroidal functions,
{\it Appl. Comput. Harmon. Anal.} {\bf 25}, No.2,  (2003),
168--176.

\bibitem{Dickinson} R. E. Dickinson, On exact and approximate linear theory of vertically propagating planetary Rossby waves forced at a spherical
lower boundary, {\it Mon. Weath. Rev.,} {\bf 96,} No. 7, (1968), 405--415.

\bibitem{Dunster} Y. M. Dunster,  Uniform asymptotic expansions for prolate spheroidal functions with large parameters,
{\it SIAM J. Math.Anal.,}  {\bf 17} No. 6, (1986),  1495--1524.

\bibitem{Flammer} C. Flammer, {\it Spheroidal Wave Functions,} Stanford Univ. Press, CA, 1957.

\bibitem{Gatteschi} L. Gatteschi,
Limitazione degli errori nelle formule asintotiche per le funzioni speciali. (Italian) Univ. e Politec. Torino. Rend. Sem. Mat.
{\bf 16}  (1957), 83--94.

\bibitem{Glasser} M. L. Glasser and M. S. Klamkin,
Some integrals of squares of Bessel functions, {\it Utilitas Math., } {\bf 12} (1977), 315--316.

\bibitem{Goss}  L. Gosse,  Compressed sensing with preconditioning for sparse recovery with subsampled matrices of Slepian prolate functions. {\sl Ann. Univ. Ferrara Sez. VII Sci. Mat.}, {\bf 59} (2013), 81--116.

\bibitem{Logan} J. A. Hogan and J. D. Lakey, {\it Duration and Bandwidth Limiting: Prolate Functions, Sampling, and Applications,}
Applied and Numerical Harmonic Analysis Series, Birkh\"aser, Springer, New York, London, 2013.

\bibitem{Karoui1} A. Karoui and T. Moumni, New efficient
methods of computing the prolate spheroidal wave functions and
their corresponding eigenvalues, {\it Appl. Comput. Harmon. Anal.}
{\bf 24}, No.3,  (2008), 269--289.




\bibitem{Landau2} H. J. Landau and H. O. Pollak, Prolate spheroidal  wave functions, Fourier analysis and
uncertainty-III. The dimension of space of essentially time-and
band-limited signals, {\it Bell System Tech. J.} {\bf 41}, (1962),
1295--1336.

\bibitem{LW} H. J. Landau and H. Widom, Eigenvalue
distribution of time and frequency limiting, {\it J. Math. Anal.Appl.,} {\bf 77,} (1980), 469--481.

\bibitem{Li} L. W. Li, X. K. Kang, M. S. Leong, Spheroidal wave functions in electromagnetic
theory, Wiley-Interscience publication, 2001.


\bibitem{Lin} W. Lin, N. Kovvali and L. Carin, Pseudospectral method based on prolate spheroidal wave functions for
semiconductor nanodevice simulation, {\it Computer Physics
Communications,} {\bf 175} (2006),  78--85.

\bibitem{Miles1} J. W. Miles, Asymptotic approximations  for prolate spheroidal wave functions, {\it Stud. Appl. Math.,}
{\bf 54,} No. 4, (1975), 315--349.


\bibitem{Miles2}  J. W. Miles, Asymptotic eigensolutions of Laplace's tidal equation, {\it Proc. R. Soc. Lond. A,} {\bf 353,} No. 1674, (1977), 377--400.

\bibitem{Muller1} R. M\"uller, On the Structure of the Global Linearized Primitive Equations Part II: Laplace's Tidal Equations,
{\it Beitr. Phys. Atmosph.,} {\bf 62,} No. 2, (1989),  112--125.

\bibitem{Muller2}  R. M\"uller, Stable and  Unstable Eigensolutions of Laplace's Tidal Equations for Zonal Wavenumber Zero,
{\it Adv. Atmos. Sci.,} {\bf 10,} No.1, (1993), 21--40.

\bibitem{Nikoforov} A. N. Nikoforov and V. B. Uvarov, Special
functions of mathematical physics, translated from the Russian
edition, Birkh\"aser Verlag Basel, (1988).

\bibitem{Niven} C. Niven, On the Conduction of Heat in Ellipsoids of
Revolution, {\it Phil. Trans. R. Soc. Lond.,}  {\bf 171,} (1880),
117-151.

\bibitem{Oconnor} W. P. Oconnor, On the application of the spheroidal wave equation to the dynamical theory of the long-period zonal tides
in a global ocean, {\it Proc. R. Soc. Lond. A,} {\bf 439,} No. 1905, (1992), 189--196.

\bibitem{Olver} F. W. J. Olver, {\it Asymptotics and Special Functions,} Academic Press, New York, 1974.

\bibitem{Osipov} A. Osipov, Certain inequalities involving prolate spheroidal wave functions and associated quantities,
{\it Appl. Comput. Harmon. Anal.}, {\bf 35},  (2013), 359--393.


\bibitem{Slepian1} D. Slepian and H. O. Pollak, Prolate spheroidal wave functions, Fourier analysis and
uncertaintyI, {\it Bell System Tech. J.} {\bf 40} (1961), 43--64.

\bibitem{Slepian3} D. Slepian,  Prolate spheroidal wave functions, Fourier analysis and
uncertainty--IV: Extensions to many dimensions; generalized
prolate spheroidal functions, {\it Bell System Tech. J.} {\bf 43}
(1964), 3009--3057.

\bibitem{Slepian2} D. Slepian, Some Asymptotic Expansions for Prolate Spheroidal Wave Functions,
{\it J. Math. Phys.,} {\bf 44}, No. 2, (1965), 99--140.

\bibitem{Szego} Szeg\"o, G\'abor, {\it  Orthogonal polynomials,} Fourth edition, American Mathematical Society,
Colloquium Publications, Vol. XXIII. American Mathematical Society, Providence, R.I., 1975.

 \bibitem{Xiao} H. Xiao, V. Rokhlin and N. Yarvin, Prolate spheroidal wave functions, quadrature and
interpolation, {\it Inverse Problems,} {\bf 17}, (2001), 805--838.

\bibitem{Walter} G. Walter and T. Soleski, A new friendly method of computing prolate spheroidal wave
functions and wavelets, {\it Appl. Comput. Harmon. Anal.} {\bf 19}, (2005), 432--443.

\bibitem{Wang} L. L. Wang,  Analysis of spectral approximations using prolate spheroidal wave functions.
 {\it Math. Comp.} {\bf 79} (2010), no. 270, 807--827.

 \bibitem{Watson} G. N. Watson, A Treatise on the Theory of Bessel Functions, Second Edition,  Cambridge University Press,
 London, New York, 1966.


 \bibitem{Widom} H. Widom, Asymptotic behavior of the eigenvalues of certain integral equations. II,
 {\it Archive for Rational Mechanics and Analysis,}
{\bf 17}  No.  (1964), 215--229

\end{thebibliography}
\end{document}